\documentclass{conm-p-l}

\usepackage{amsmath}
\usepackage{amsthm}
\usepackage{amssymb}
\usepackage{euscript}
\usepackage{diagrams}

\theoremstyle{plain}
\newtheorem{thm}[subsection]{Theorem}
\newtheorem{prp}[subsection]{Proposition}
\newtheorem{lem}[subsection]{Lemma}
\newtheorem{cor}[subsection]{Corollary}
\newtheorem{dfn}[subsection]{Definition}

\theoremstyle{remark}
\newtheorem{exm}[subsection]{Example}
\newtheorem{exms}[subsection]{Examples}
\newtheorem{rmk}[subsection]{Remark}

\def\HCC{\mathrm{CC}}
\def\HH{\mathrm{HH}}
\def\LL{\mathcal{L}}
\def\op{\mathrm{op}}
\def\CC{\mathcal{C}}
\def\DD{\mathcal{D}}
\def\uCC{\underline{\mathcal{C}}}
\def\Top{\mathrm{Top}}
\def\uTop{\underline{\Top}}
\def\OO{\mathcal{O}}
\def\Hom{\mathrm{Hom}}
\def\uHom{\underline{\mathrm{\Hom}}}
\def\Coend{\mathrm{Coend}}
\def\End{\mathrm{End}}
\def\EE{\mathcal{E}}
\def\uEE{\underline{\mathcal{E}}}
\def\al{\alpha}
\def\be{\beta}
\def\Sg{\Sigma}
\def\sg{\sigma}
\def\lrto{\longrightarrow}
\def\lra{\leftrightarrows}
\def\FF{\mathcal{F}}
\def\Cat{\mathbf{Cat}}

\def\Aa{\mathcal{A}}
\def\Bb{\mathcal{B}}
\def\Alg{\mathrm{Alg}}
\def\NN{\mathbb{N}}
\def\RR{\mathbb{R}}
\def\BB{\mathbb{B}}
\def\RB{\mathbb{RB}}
\def\RS{\mathbb{R}\Sg}
\def\Aut{\mathrm{Aut}}
\def\inc{\hookrightarrow}
\def\ZZ{\mathbb{Z}}
\def\Set{\mathrm{Set}}
\def\Ass{\mathrm{Ass}}
\def\Ch{\mathrm{Ch}}
\def\Fin{\mathcal{F}\mathit{in}}
\def\Mod{\mathrm{Mod}}

\def\XX{\mathcal{X}}
\def\PP{\mathcal{P}}
\def\tam{\EuScript Tam}
\def\cgr{\EuScript K}
\def\colim{\mathrm{colim}}
\def\hocolim{\mathrm{hocolim}}

\begin{document}

\title{The lattice path operad and Hochschild cochains}
\author{M. A. Batanin}
\address{Macquarie University, NSW 2109, Australia}
\email{mbatanin@ics.mq.edu.au}

\author{C. Berger}
\address{Universit\'e de Nice, Lab. J.-A. Dieudonn\'e, Parc Valrose, 06108 Nice Cedex 2, France}
\email{cberger@math.unice.fr}


\date{2 February 2009}


\begin{abstract}We introduce two coloured operads in sets -- the lattice path operad and a cyclic extension of it -- closely related to iterated loop spaces and to universal operations on cochains. As main application we present a formal construction of an $E_2$-action (resp. framed $E_2$-action) on the Hochschild cochain complex of an associative (resp. symmetric Frobenius) algebra.\end{abstract}

\maketitle

\section*{Introduction}

The algebraic structure of iterated loop spaces is best captured by the action of the operad of little $n$-cubes or any other equivalent \emph{$E_n$-operad} \cite{BV,Ma}. Over the last fifteen years, \emph{Deligne's conjecture} motivated the research of cellular models of the little disks operad, suitable to act on the Hochschild cochain complex of an associative algebra. According to the \emph{cyclic} Deligne conjecture \cite{TT} this action extends to a cellular model of the \emph{framed} little disks operad if the algebra is a symmetric Frobenius algebra. Both conjectures have been given several proofs, see for instance \cite{BF,Ka1,Ka2,KS1,KS2,MS1,Ta,TZ,V,V2}. The resulting (framed) $E_2$-action on Hochschild cochains is a chain-level realization of the Gerstenhaber (resp. Batalin-Vilkovisky) algebra structure of the Hochschild cohomology of an associative (resp. symmetric Frobenius) algebra, cf. \cite{Ger,Get,Me}.

Our purpose here is to give a parallel and conceptually simple proof of both conjectures.  We extend McClure-Smith's cosimplicial techniques \cite{MS3} making use of crossed simplicial groups, especially of Connes' cyclic category \cite{Co,FT,FL,Kr}. The novelty of our approach is the construction of two filtered coloured operads in sets -- the \emph{lattice path operad} $\LL$ and the \emph{cyclic lattice path operad} $\LL^{cyc}$ -- which encode the combinatorial structure of iterated loop spaces in a ``context-independent'' way. In particular, the combinatorial structure inducing a (framed) $E_2$-action is completely transparent in our setting, and the (cyclic) Deligne conjecture follows from the combination of the following four facts:\begin{itemize}\item[(i)]the second filtration stage $\LL_2$ (resp. $\LL_2^{cyc}$) acts on multiplicative non-symmetric (resp. multiplicative cyclic) operads;\item[(ii)] the endomorphism operad of an algebra (resp. symmetric Frobenius algebra) is a multiplicative non-symmetric (resp. multiplicative cyclic) operad;\item[(iii)]the Hochschild cochain complex is obtained from the endomorphism operad by conormalization;\item[(iv)]condensation of $\LL_2$ (resp. $\LL_2^{cyc}$) with respect to the standard cosimplicial (resp. cocyclic) object in chain complexes yields an $E_2$-chain operad (resp. framed $E_2$-chain operad).\end{itemize}Property (iv) holds in any closed symmetric monoidal category equipped with a ``good'' cosimplicial (resp. cocyclic) object and a compatible Quillen model structure. In the category of topological spaces our proof recovers McClure-Smith's \cite{MS3} (resp. Salvatore's \cite{Sa}) topological version of the (cyclic) Deligne conjecture.\vspace{1ex}

Here is an outline of the article:\vspace{1ex}

 In Section \ref{cond} we review some basic constructions involving operads. We introduce the term \emph{condensation} for the combination of two in the literature existing constructions: a convolution product for coloured operads due to Day-Street \cite{DS}, followed by a generalized coendomorphism operad for diagrams due to McClure-Smith \cite{MS3}. Condensation takes a pair $(\OO,\delta)$ consisting of a coloured operad $\OO$ and a diagram $\delta$ on the unary part of $\OO$, to a single-coloured operad $\Coend_\OO(\delta)$ in the target category of $\delta$. Condensation of coloured operads is the main technical tool of this article and seems to be of independent interest.

In Section \ref{latticepath} we introduce the lattice path operad $\LL$. The unary part of $\LL$ coincides with the simplex category $\Delta$. The lattice path operad comes equipped with an operadic filtration by complexity. We characterize the categories of algebras over the $k$-th filtration stage $\LL_k$ for $k=0,1,2.$ These are respectively the categories of cosimplicial objects, of cosimplicial $\square$-monoids, and of multiplicative non-symmetric operads. The simplicial $n$-sphere $\Delta[n]/\partial\Delta[n]$ is an $\LL_n$-coalgebra in the category of finite pointed sets. This recovers a result of Sinha \cite{Si} concerning the structure of the simplicial $2$-sphere. Moreover, we obtain a canonical $E_\infty$-action on the normalized cochains of a simplicial set (cf. \cite{BF,MS2}) as well as a canonical $E_n$-action on Pirashvili's $n$-th order higher Hochschild cochains (cf. \cite{Gi,Pi}).

In Section \ref{complexity} we study the homotopy type of the condensation of $\LL_n.$ We recall the definition of the complete graph operad $\cgr$ and its filtration \cite{Be2}, and show that the complexity index induces a filtration-preserving morphism of operads from the lattice path operad $\LL$ to the complete graph operad $\cgr$. From this we deduce our main theorem that $\delta$-condensation of $\LL_n$ in a monoidal model category yields an $E_n$-operad provided that the cosimplicial object $\delta$ interacts well with the model structure and with $\LL$ (we call $\LL$ strongly $\delta$-reductive in this case). For $n=2$ this implies Deligne's conjecture according to the aforementioned scheme.

In Section \ref{cyclic} we introduce the cyclic lattice path operad $\LL^{cyc}$. The unary part of $\LL^{cyc}$ coincides with the cyclic category of Connes \cite{Co}. We prove that condensation of $\LL_2^{cyc}$ with respect to the standard cocyclic object in topological spaces (resp. chain complexes) yields a framed $E_2$-operad in topological spaces (resp. chain complexes). Again, according to the aforementioned scheme, this readily implies the cyclic Deligne conjecture.\vspace{1ex}

Let us comment on some perspectives of future work:\vspace{1ex}

There is a close relationship between symmetric Frobenius algebras and Calabi-Yau categories. According to Costello \cite{Cos}, the latter give rise to open Topological Conformal Field Theories, and a framed $E_2$-action on the Hochschild cochain complex is just the genus $0$ part of such a TCFT-action. This suggests connections between the cyclic lattice path operad, the moduli space of Riemann spheres, and Chas-Sullivan's theory of string topology operations. In a forthcoming series of papers \cite{BBM}, the relationship of the (cyclic) lattice path operad with different versions of the surjection operad \cite{BF,MS2}, with the operad of natural operations on Hochschild cochains \cite{BMl}, and with Sullivan's chord diagrams \cite{TZ} will be investigated.

Our method applies to other monoidal model categories than topological spaces or chain complexes, such as the category  $\Cat$ of small categories with the Joyal-Tierney model structure or the category of $2$-categories with Lack's model structure. The case $\Cat$ is already interesting. Condensation of $\LL_2$ with respect to the standard cosimplicial object (consisting in degree $n$ of a contractible groupoid on $n+1$ objects) yields a categorical $E_2$-operad. We get in particular an $E_2$-action (i.e. a braided monoidal structure) on the codescent object of the Hochschild complex of a monoidal category. The latter coincides with Joyal-Street's \cite{JS} \emph{center} of the monoidal category. We expect a similar result in the case of $2$-categories which would considerably simplify the approach of Baez-Neuchl \cite{BN}.

An intriguing problem is the determination of higher operadic structures hidden in the lattice path operad. In fact, we discovered the lattice path operad when trying to understand and generalize Tamarkin's \cite{Ta2} construction of a contractible $2$-operad acting on the derived $2$-category of $dg$-categories. To be more precise, by \cite{Ba3}, the $n$-th filtration stage of the complete graph operad contains an internal $n$-operad $a_{n}$. The comma categories $\{\cgr_n/a_T\}_{T\in Ord(n)}$ form therefore a categorical $n$-operad  ${\cgr_n/a_n}$ which comes equipped with an $n$-operadic functor ${\cgr_n/a_n}\rightarrow Des_n(\cgr_n)$, where $Des$ denotes the desymmetrisation functor of \cite{Ba3}. Define a coloured categorical $n$-operad $\tam_n$ by the following pullback\begin{diagram}[small]\tam_n&\rTo&\cgr_n/a_n\\\dTo&&\dTo\\ Des_n(\LL_n)&\rTo&Des_n(\cgr_n).\end{diagram}Condensation of $\tam_n$ produces then a contractible $n$-operad. For $n=2$, this is precisely Tamarkin's contractible $2$-operad which acts on the derived $2$-category of $dg$-categories. In combination with \cite{Ba2} this action gives a particularly clear proof of Deligne's conjecture. We hope to be able to extend this line of proof to higher-dimensional versions of the Deligne conjecture as formulated by Kontsevich.\vspace{3ex}

\emph{Acknowledgements:} The authors are grateful to Denis-Charles Cisinski, Ezra Getzler, Andr\'e Joyal, Dmitry Kaledin, Joachim Kock, Martin Markl, Joan Mill\`es, Paolo Salvatore, Ross Street, Dmitry Tamarkin and Mark Weber for many illuminating discussions. The helpful comments of the anonymous referee have been much appreciated. The first author gratefully acknowledges financial support of the Australian Research Council and Scott Russel Johnson Foundation. The second author gratefully acknowledges financial support of the French CNRS and the French Agence Nationale de Recherche (grant OBTH). Both authors thank the CRM of Barcelona for the stimulating atmosphere offered by the special research program 2007/2008 on Derived Categories and Higher Homotopy Structures.

\section{\label{cond} Condensation of coloured operads}

\subsection{Coendomorphism operads}Let $\EE=(\EE,\otimes_\EE,I_\EE,\tau_\EE)$ be a \emph{closed symmetric monoidal category} with tensor $\otimes_\EE$, unit $I_\EE$, symmetry $\tau_\EE$, and internal hom $\uEE(-,-)$. The subscript $\EE$ will often be omitted. For an arbitrary $\EE$-category $\CC$, let $\CC^{\otimes k}$ be the $\EE$-category with objects the $k$-tuples of objects of $\CC$, and hom-objects $$\uCC^{\otimes k}((X_i)_{1\leq i\leq k},(Y_i)_{1\leq i\leq k})=\bigotimes_{i=1}^k\uCC(X_i,Y_i).$$The iterated tensor product $$\xi_k^\otimes:\EE^{\otimes k}\to\EE:(X_1,\dots,X_k)\mapsto X_1\otimes\cdots\otimes X_k$$ is an \emph{$\EE$-functor}, i.e. there are canonical morphisms $$\uEE(X_1,Y_1)\otimes\cdots\otimes\uEE(X_k,Y_k)\to\uEE(X_1\otimes\cdots\otimes X_k,Y_1\otimes\cdots\otimes Y_k).$$This permits the definition of a \emph{coendomorphism operad} $\Coend(X)$ of $X$ by$$\Coend(X)(k)=\uEE(X,X^{\otimes k}),\quad k\geq 0,$$where the symmetric group $\Sg_k$ on $k$ letters acts on $\Coend(X)(k)$ by permutation of the tensor factors, and the operad substitution maps$$\Coend(X)(k)\otimes\Coend(X)(i_1)\otimes\cdots\otimes\Coend(X)(i_k)\to\Coend(X)(i_1+\cdots+i_k)$$are given by tensoring the last $k$ factors and precomposing with the first factor.

\begin{prp}\label{coend1}Let $X,Y$ be objects of $\EE$ and assume that $Y$ is a commutative monoid in $\EE$. Then $\uEE(X,Y)$ has a canonical $\Coend(X)$-algebra structure.\end{prp}

\begin{proof}The action is given by $$\Coend(X)(k)\otimes\uEE(X,Y)^{\otimes k}\to\Coend(X)(k)\otimes\uEE(X^{\otimes k},Y^{\otimes k})\to\uEE(X,Y^{\otimes k})\to\uEE(X,Y)$$ where the first map is induced by an iterated tensor, the second map by composition and the third map by the commutative monoid structure of $Y$.\end{proof}

\subsection{Functor-operads}For any $\EE$-functor $\xi:\CC^{\otimes k}\to\CC$ and permutation $\sg\in\Sg_k$, we shall write $\xi^\sg$ for the $\EE$-functor $$\xi^\sg(X_1,\dots,X_k)=\xi(X_{\sg^{-1}(1)},\dots,X_{\sg^{-1}(k)}).$$ In particular, $\xi^{\sg\tau}=(\xi^\sg)^\tau$. An $\EE$-functor $\xi:\CC^{\otimes k}\to\CC$ is called \emph{twisted-symmetric} if $\xi$ comes equipped with $\EE$-natural transformations $\phi_\sg:\xi\to\xi^\sg,\,\sg\in\Sg_k,$ such that $\phi_{\sg\tau}=(\phi_\sg)^\tau\phi_\tau$ for $\sg,\tau\in\Sg_k$, and such that $\phi_e$ is the identity transformation where $e$ denotes the neutral element of $\Sg_k$.

\begin{dfn}[\cite{MS3}]\label{functor-operad}A \emph{functor-operad} $\xi$ on an $\EE$-category $\CC$ consists of a sequence of twisted-symmetric $\EE$-functors $\xi_k:\CC^{\otimes k}\to\CC,\,k\geq 0,$ together with $\EE$-natural transformations $$\mu_{i_1,\dots,i_k}:\xi_k\circ(\xi_{i_1}\otimes\cdots\otimes\xi_{i_k})\to\xi_{i_1+\cdots+i_k},\quad i_1,\dots,i_k\geq 0,$$such that\begin{enumerate}\item[(i)]$\xi_1$ is the identity functor and $\xi_k\circ(\xi_1\otimes\cdots\otimes\xi_1)\overset{\mu_{1,\dots,1}}{=}\xi_k\overset{\mu_k}{=}\xi_1\circ\xi_k$;\item[(ii)] the $\mu_{i_1,\dots,i_k}$ are associative;\item[(iii)] all diagrams of the following form commute:\begin{diagram}[small]\xi_k\circ(\xi_{i_1}\otimes\cdots\otimes\xi_{i_k})&\rTo^{\mu_{i_1,\dots,i_k}}&\xi_{i_1+\cdots+i_k}\\\dTo^{\phi_\sg\circ(\phi_{\sg_1}\otimes\cdots\otimes\phi_{\sg_k})}&&\dTo_{\phi_{\sg(\sg_1,\dots,\sg_k)}}\\\xi^\sg_k\circ(\xi^{\sg_1}_{i_1}\otimes\cdots\otimes\xi^{\sg_k}_{i_k})&\rTo^{\mu_{i_1,\dots,i_k}}&\xi^{\sg(\sg_1,\dots,\sg_k)}_{i_1+\cdots+i_k},\end{diagram}where $\sg(\sg_1,\dots,\sg_k)$ is the value of $(\sg;\sg_1,\dots,\sg_k)\in\Sg_k\times\Sg_{i_1}\times\cdots\times\Sg_{i_k}$ under the \emph{permutation operad}, cf. (\ref{string}).\end{enumerate}\vspace{1ex}%
A \emph{$\,\xi$-algebra} is an object $X$ of $\,\CC$ equipped with a sequence of morphisms $$\al_k:\xi_k(X,\dots,X)\to X,\quad k\geq 0,$$ such that\begin{enumerate}\item[(i)]$\al_1=1_X$;\item[(ii)]$\al_k\circ\phi_\sg^{X,\dots,X}=\al_k$ for all $\sg\in\Sg_k$;\item[(iii)] for all $i_1,\dots,i_k\geq 0$ the following diagram commutes in $\CC$:\begin{diagram}[small]\xi_k(\xi_{i_1}(X,\dots,X),\dots,\xi_{i_k}(X,\dots,X))&\rTo^{\mu_{i_1,\dots,i_k}}&\xi_{i_1+\cdots+i_k}(X,\dots,X)\\\dTo^{\xi_k(\al_{i_1},\dots,\al_{i_k})}&&\dTo_{\al_{i_1+\cdots+i_k}}\\\xi_k(X,\dots,X)&\rTo^{\al_k}&X.\end{diagram}
\end{enumerate}
\end{dfn}
\vspace{1ex}

\begin{prp}[\cite{MS3}]\label{coend2}Let $X,Y$ be objects of an $\EE$-category $\CC$ with functor-operad $\xi$, and let $Y$ be a $\xi$-algebra. Then $\uCC(X,Y)$ is a $\Coend_\xi(X)$-algebra where the coendomorphism operad is given by $\Coend_\xi(X)(k)=\uCC(X,\xi_k(X,\dots,X)),\,k\geq 0$.\end{prp}

\begin{proof}We shall write $\xi_k(X)$ for $\xi_k(X,\dots,X)$. Since $\xi_k$ is an $\EE$-functor, the $\EE$-natural tranformation $\mu_{i_1,\dots,i_k}$ induces for objects $X,Y,Z$ of $\CC$ a substitution map $$\mu_{i_1,\dots,i_k}^{X,Y,Z}:\uCC(X,\xi_k(Y))\otimes\uCC(Y,\xi_{i_1}(Z))\otimes\cdots\otimes\uCC(Y,\xi_{i_k}(Z))\to\uCC(X,\xi_{i_1+\cdots+i_k}(Z)).$$In virtue of the defining properties of the functor-operad $\xi$, these substitution maps satisfy the unit, associativity and equivariance properties of an operad substitution map in $\EE$; in particular, the symmetric sequence $\Coend_\xi(X)(k),\,k\geq 0,$ is indeed an operad in $\EE$, and the symmetric sequence $\uCC(X,\xi_k(Y)),\, k\geq 0,$ is a left module over $\Coend_\xi(X)$. It follows then from the defining properties of the $\xi$-algebra $Y$ that\begin{diagram}[small]\Coend_\xi(X)(k)\otimes\uCC(X,Y)^{\otimes k}&\rTo^{\mu_{1,\dots,1}^{X,X,Y}}&\uCC(X,\xi_k(Y))&\rTo^{\uCC(X,\al_k)}&\uCC(X,Y)\end{diagram}defines the asserted operad action.\end{proof}

\begin{rmk}\label{other}Any closed symmetric monoidal category $\EE$ is an $\EE$-category, and the assignment $\xi^\otimes_k(X_1,\dots,X_k)=X_1\otimes\cdots\otimes X_k$ extends to a functor-operad $\xi^\otimes$. Commutative monoids in $\EE$ are then precisely $\xi^\otimes$-algebras; therefore, Proposition \ref{coend1} is a special case of Proposition \ref{coend2}. A functor-operad on $\CC$ is a particular instance of an \emph{internal operad} inside the categorical endomorphism operad $\underline{\mathrm{End}}_\EE(\CC)$; the general concept of an internal operad inside a categorical operad has been introduced by the first-named author in \cite{Ba3}; another closely related concept has been studied by Day-Street \cite{DS} under the name \emph{symmetric lax monoidal structure}. It is possible to dualize Proposition \ref{coend2} in order to generalize the obvious dual of Proposition \ref{coend1}. This however requires to replace functor-operads by functor-cooperads (resp. symmetric lax monoidal structures by symmetric colax monoidal structures).\end{rmk}

\subsection{Coloured operads} Let $N$ be a set of ``colours''. Recall that an $N$-coloured operad $\OO$ in $\EE$ consists of objects in $\EE$ $$\OO(n_1,\dots,n_k;n),\quad (n_1,\dots,n_k;n)\in N^{k+1},\,k\geq 0,$$together with units $I_\EE\to\OO(n,n)$ and substitution maps$$\OO(m_1,\dots,m_l;n_i)\otimes\OO(n_1,\dots,n_k;n)\overset{\circ_i}{\to}\OO(n_1,\dots,n_{i-1},m_1,\dots,m_l,n_{i+1},\dots,n_k;n)$$which fullfill natural unit, associativity and equivariance axioms. The colours $n_1,\dots,n_k$ are called \emph{input colours}, while the colour $n$ is called \emph{output colour}.

If $N=\{*\}$, we recover the classical concept of an operad by setting $$\OO(k)=\OO(\overbrace{*,\dots,*}^k;*),\quad k\geq 0.$$ If $N$ is not a set, but a proper class, the term \emph{multicategory} is more appropriate than that of a coloured operad; in other words, a coloured operad is precisely a \emph{small} multicategory. This smallness condition is essential for our purpose. The use of ``colours'' in this context goes back to Boardman-Vogt's seminal book \cite{BV}.

The \emph{underlying category} $\OO_u$ of a coloured operad $\OO$ has the colours as objects and the \emph{unary operations} as morphisms, i.e. $$\OO_u(m,n)=\OO(m;n).$$If $\OO$ is a coloured operad in $\EE$, then $\OO_u$ is an $\EE$-category. Since the unary operations act contravariantly on the inputs and covariantly on the output, any coloured operad $\OO$ in $\EE$ can be considered as a sequence of functors$$\OO(\overbrace{-,\dots,-}^k;-):\overbrace{\OO_u^\op\otimes\cdots\otimes\OO_u^\op}^k\otimes\OO_u\to\EE,\quad k\geq 0.$$
The category of $\EE$-functors $\OO_u\to\EE$ and $\EE$-natural transformations is the underlying category of an $\EE$-category which we shall denote by $\EE^{\OO_u}$. According to Day-Street \cite{DS}, each coloured operad $\OO$ in $\EE$ induces a sequence of $\EE$-functors $$\xi(\OO)_k:\overbrace{\EE^{\OO_u}\otimes \cdots \otimes \EE^{\OO_u}}^k  \lrto \EE^{\OO_u},\quad k\geq 0,$$ by the familiar coend formulas$$\xi(\OO)_k(X_1,\ldots,X_k)(n)= \OO(-,\dots,-;n)\otimes_{\OO_u\otimes\cdots\otimes\OO_u}X_1(-)\otimes\cdots\otimes X_k(-).$$

\begin{prp}[\cite{DS}]\label{convolution}The sequence $\xi(\OO)_k,\,k\geq 0,$ extends to a functor-operad on the diagram category $\EE^{\OO_u}$ in such a way that the categories of $\OO$-algebras and of $\xi(\OO)$-algebras are canonically isomorphic.\end{prp}

\begin{proof}The twisted symmetry of $\xi(\OO)_k$ follows from the $\Sg_k$-actions on the $k$-ary operations of $\OO$. The unit conditions for $\xi(\OO)$ are a consequence of the Yoneda lemma. Existence, associativity and equivariance of the substitution maps of $\xi(\OO)$ follow from those of $\OO$.

An $\OO$-algebra $X$ is a family $X(n),\,n\in N,$ of objects of $\EE$ together with unital, associative and equivariant action maps$$\OO(n_1,\dots,n_k;n)\otimes X(n_1)\otimes\cdots\otimes X(n_k)\lrto X(n),\quad (n_1,\dots,n_k;n)\in N^{k+1},\,k\geq 0.$$In particular, $X$ extends to an $\EE$-functor $\OO_u\to\EE$. The universal property of coends implies that the action maps above correspond bijectively to a sequence of maps$$\xi(\OO)_k(X,\dots,X)\lrto X,\quad k\geq 0,$$ endowing $X$ with the structure of a $\xi(\OO)$-algebra.\end{proof}

\subsection{Condensation and totalization}\label{condensation}Assembling Propositions \ref{convolution} and \ref{coend2} into one construction takes a coloured operad to a single-coloured operad $$\OO\mapsto\Coend_{\xi(\OO)}(\delta)$$ for every choice of $\EE$-functor $\delta:\OO_u\to\EE$. We call this composite construction \emph{$\delta$-condensation}, since -- intuitively speaking -- a diagram $\delta$ on the unary part of a coloured operad allows one to ``condense'' the given set of colours into a single colour. For notational ease we shall write $\Coend_\OO(\delta)$ instead of $\Coend_{\xi(\OO)}(\delta)$. By Proposition \ref{coend2} there is a companion \emph{$\delta$-totalization} functor $$\uHom_{\OO_u}(\delta,-):\Alg_\OO\lrto\Alg_{\Coend_\OO(\delta)}$$which takes an $\OO$-algebra $A$ to a $\Coend_{\OO}(\delta)$-algebra $\uHom_{\OO_u}(\delta,A)$. The bifunctor $\uHom_{\OO_u}$ denotes the $\EE$-valued hom of the $\EE$-category $\EE^{\OO_u}$.\vspace{1ex}

\begin{exm}Consider the case where $\OO$ is already a single-coloured operad, say in $k$-modules for a unital commutative ring $k$. The underlying category $\OO_u$ of unary operations is then a one-object $k$-linear category with $\OO_u(*,*)=\OO(1)$. Therefore, a diagram $\delta$ on $\OO_u$ is nothing but an $\OO(1)$-module, and $\delta$-condensation takes the given operad $\OO$ in $k$-modules to an operad $\Coend_\OO(\delta)$ in $\OO(1)$-modules. The unary part of the condensed operad is by definition $\uHom_{\OO(1)}(\delta,\delta)$. In the special case $\delta=\OO(1)$, the condensed operad in $\OO(1)$-modules has as object of unary operations the unit $\OO(1)$ of the category of $\OO(1)$-modules. In other words, even in this simple case, $\delta$-condensation is an interesting construction, which allows one to get rid of the unary operations through an extension of the base ring.\end{exm}

\section{\label{latticepath}The lattice path operad}

Recall that the \emph{$n$-th ordinal} $[n]$ is defined to be the category freely generated by the linear graph
$l_n=(0\rightarrow 1 \rightarrow \ldots \rightarrow  n)$ of length $n$. We define the \emph{tensor product of ordinals} $[m]\otimes[n]$ to be the category freely generated by the $(m,n)$-grid $l_m\otimes l_n$. The latter has as vertices the pairs $(i,j)$ for $0\leq i\leq m$ and $0\leq j\leq n$, and as only edges those $(i,j)\to(i',j')$ for which $(i',j')=(i+1,j)$ or $(i',j')=(i,j+1)$. This tensor product of ordinals extends (by Day convolution \cite{Da}) to a closed symmetric monoidal structure on the category $\Cat$ of small categories; often, the resulting tensor product of categories is called the \emph{funny} tensor product of categories in order to distinguish it from the cartesian product of categories.

We shall also consider the category $\Cat_{*,*}$ consisting of \emph{bipointed} small categories and functors preserving the two distinguished objects. For categories $\Aa,\Bb$, bipointed respectively by $(a_0,a_1)$ and $(b_0,b_1)$, the tensor product $\Aa\otimes\Bb$ is again bipointed by $((a_0,b_0),(a_1,b_1))$. We consider the ordinals $[n]$ as bipointed by $(0,n)$.

\begin{dfn}The \emph{lattice path operad} $\LL$ is the $\NN$-coloured operad in sets with $$\LL(n_1,\ldots, n_k; n)=\Cat_{*,*}([n+1],[n_1 +1]\otimes\cdots\otimes[n_k +1]),$$the operad substitution maps being induced by tensor and composition in $\Cat_{*,*}$.\end{dfn}

\subsection{Lattice paths as integer-strings}\label{string}A lattice path $x\in\LL(n_1,\dots,n_k;n)$ is a functor $[n+1]\to[n_1+1]\otimes\cdots\otimes[n_k+1]$ which takes $0$ to $(0,\dots,0)$ and $n+1$ to $(n_1+1,\dots,n_k+1)$, and which consists of $n+1$ morphisms $x(0)\to x(1)\to\cdots\to x(n)\to x(n+1)$. A morphism in $[n_1+1]\otimes\cdots\otimes[n_k+1]$ corresponds to a finite edge-path in the $(n_1+1,\dots,n_k+1)$-grid $l_{n_1+1}\otimes\cdots\otimes l_{n_k+1}$. Each edge in $l_{n_1+1}\otimes\cdots\otimes l_{n_k+1}$ is determined by its initial vertex together with the choice of a coordinate-axis, i.e. a direction among $k$ possible directions.

In conclusion, a lattice path $x\in\LL(n_1,\dots,n_k;n)$ determines, and is determined by, an \emph{integer-string} $(a_i)_{1\leq i\leq n_1+\cdots+n_k+k}$ which contains $n_1+1$ times the integer $1$, $n_2+1$ times the integer $2$, \dots, $n_k+1$ times the integer $k$, and which is \emph{subdivided} into $n+1$ substrings according to the $n+1$ morphisms $x(i-1)\to x(i)$. Empty substrings represent identity morphisms. Such a partition of an integer-string into $n+1$ substrings will be represented by an \emph{insertion of $\,n$ vertical bars} such that each substring is delimited either by two consecutive bars or, if the substring is initial (resp. terminal), by the leftmost (resp. rightmost) bar. For instance, the subdivided integer-string $1|12|21$ corresponds to the following lattice path $x\in\LL(2,1;2)$:

\begin{diagram}[small](0,2)&\hLine&(1,2)&\hLine&\cdot&\rTo^1&x(3)\\\vLine&&\vLine&&\uTo^2&&\vLine\\(0,1)&\hLine&(1,1)&\hLine&x(2)&\hLine&(3,1)\\\vLine&&\vLine&&\uTo^2&&\vLine\\x(0)&\rTo^1&x(1)&\rTo^1&\cdot&\hLine&(3,0)\end{diagram}

Under this identification, the substitution maps of the lattice path operad are given by \emph{renumbering} and \emph{string substitution}. For instance, we have $$1||12|3|2\circ_23|12=1'||1'2|5'|2\circ_{2}4'|2'3'=1'||1'4'|5'|2'3',$$ where we have separated the renumbering $(13123)\mapsto(1'5'2'3'4')$ from the substitution $(4'|2'3')\mapsto(2|2).$ In particular, the endomorphism operad of colour $0$ inside $\LL$ coincides with the so-called \emph{permutation operad}:$$\Sg_k=\LL(\overbrace{0,\dots,0}^k;0),\quad k\geq 0.$$ The reader should be aware of the fact that the definition of the permutation operad depends on the side from which $\Sg_k$ acts on itself. In this article we adopt left actions, i.e. a permutation $\sg\in\Sg_k$ is identified with the integer-string $\sg(1)\cdots\sg(k)$.

\subsection{The category $\Delta$}The full subcategory of $\Cat$ spanned by the ordinals $[n],\,n\geq 0,$ is called the \emph{simplex category} and denoted $\Delta$. As usual, a covariant functor $\Delta\to\EE$ is called a \emph{cosimplicial object} in $\EE$, and a contravariant functor $\Delta^\op\to\EE$ is called a \emph{simplicial object} in $\EE$. The representable presheaf $\Delta(-,[n])$ will be denoted $\Delta[n]$. Its boundary $\partial\Delta[n]$ consists of those ``simplices'' $[m]\to[n]$ which factor through a non-identity monomorphism in $\Delta$.

Adding an initial object to $\Delta$ (the ``empty'' ordinal $[-1]$) defines the category $\Delta_+$. This augmented simplex category is a \emph{monoidal category} with respect to join (juxtaposition) $[m]*[n]=[m+n+1]$. Actually, in Mac Lane's \cite{ML} terminology, $\Delta_+$ is the PRO for associative monoids, i.e. for any monoidal category $(\EE,\otimes,I_\EE)$, monoids in $\EE$ correspond bijectively to strong monoidal functors $\Delta_+\to\EE$.

\subsection{Joyal-duality}\label{duality}There is a contravariant bijection $$\Cat_{*,*}([n+1],[m+1])\cong\Cat([m],[n])$$ compatible with composition. Indeed, according to (\ref{string}), endpoint-preserving functors $\phi:[n+1]\to[m+1]$ are represented by constant integer-strings of length $m+1$, subdivided into $n+1$ substrings. Such an integer-string determines, and is determined by, a map of ordinals $\psi:[m]\to[n]$. More precisely, $\phi$ and $\psi$ determine each other by the formulas: $\psi(i)+1=\min\{j\,|\,\phi(j)>i\}$ and $\phi(j)-1=\max\{i\,|\,\psi(i)<j\}$. This duality is often referred to as \emph{Joyal-duality} \cite{Jo}.

\begin{lem}\label{Joyal}The underlying category of the lattice path operad is $\,\Delta$.\end{lem}

\begin{proof}$\LL_u(m,n)=\Cat_{*,*}([n+1],[m+1])=\Cat([m],[n])=\Delta([m],[n]).$\end{proof}

\subsection{The category $\Delta\Sigma$}\label{noncomm}There is a symmetric monoidal version $\Delta\Sigma_+$ of $\Delta_+$; in Mac Lane's \cite{ML} terminology, $\Delta\Sigma_+$ is the PROP for associative monoids, i.e. for any symmetric monoidal category $(\EE,\otimes,I_\EE,\tau_\EE)$, monoids in $\EE$ correspond bijectively to strong symmetric monoidal functors $\Delta\Sigma_+\to\EE$. The category $\Delta\Sigma$ has been described at several places. We follow Feigin-Tsygan \cite{FT}, Krasauskas \cite{Kr} and Fiedorowicz-Loday \cite{FL} who consider $\Delta\Sg$ in the context of \emph{crossed simplicial groups}.

By definition, a morphism $\phi\in\Delta\Sg([m],[n])$ consists of a map of finite sets $\phi:\{0,\dots,m\}\to\{0,\dots,n\}$ which comes equipped with an $(n+1)$-tuple of \emph{total orderings}, one for each fiber $\phi^{-1}(j),\,0\leq j\leq n,$ see for instance \cite[A10]{FT}. The composition law is defined in an obvious manner. The simplex category $\Delta$ embeds in $\Delta\Sg$ as the subcategory of order-preserving maps endowed with the natural orderings of the fibers. Again, we add an initial object $[-1]$ in order to get a symmetric monoidal category $(\Delta\Sigma_+,*,[-1],\tau)$. The symmetry $\tau$ is defined by the obvious switch map. There are canonical monomorphisms $\phi_i:[n_i]\to[n_1]*\cdots *[n_k],\,i=1,\dots,k$. For any $x:[n_1]*\cdots *[n_k]\to[n]$ in $\Delta\Sg_+$, we write $x_i:[n_i]\to[n]$ for the restriction of $x$ along $\phi_i$. Observe that $x$ is \emph{not} determined by its components $(x_1,\dots,x_k)$, i.e. the join is \emph{not} a coproduct for $\Delta\Sigma_+$.

\begin{prp}\label{duallattice}The lattice path operad has the following dual description:$$\LL(n_1,\dots,n_k;n)\cong\{x\in\Delta\Sigma_+([n_1]*\cdots*[n_k],[n])\,|\,x_i\in\Delta([n_i],[n]),\,i=1,\dots,k\},$$where the operad substitution maps are given by join and composition in $\Delta\Sigma_+$.\end{prp}

\begin{proof}A morphism $\hat{x}:[n_1]*\cdots*[n_k]\to[n]$ in $\Delta\Sg_+$ such that $\hat{x}_i\in\Delta([n_i],[n])$ can be considered as a $k$-tuple of simplices $(\hat{x}_1,\dots,\hat{x}_k)\in\Delta([n_1],[n])\times\cdots\times\Delta([n_k],[n])$ together with a total ordering of the vertex-set $[n_1]*\cdots*[n_k]$ which is compatible (under $\hat{x}$) with the natural ordering of the vertex-set $[n]$ of $\Delta[n]$.

On the other hand, there are canonical projection functors $$p_i:[n_1+1]\otimes\cdots\otimes[n_k+1]\to[n_i+1],\quad 1\leq i\leq k,$$and a lattice path $x:[n+1]\to[n_1+1]\otimes\cdots\otimes[n_k+1]$ can be considered as a $k$-tuple of projected lattice paths $$(p_1\circ x,\dots,p_k\circ x)\in\Cat_{*,*}([n+1],[n_1+1])\times\cdots\times\Cat_{*,*}([n+1],[n_k+1])$$together with a total ordering of the $(n_1+1)+\cdots+(n_k+1)$ individual steps, cf. the integer-string representation (\ref{string}) of $x$.

Therefore, Joyal-duality (\ref{duality}) establishes a canonical one-to-one correspondence between the two. We leave it to the reader to check that this correspondence respects the operad structures.\end{proof}

\begin{rmk}\label{pullback}The dual description of the lattice path operad can be interpreted as a pullback in the category of coloured operads. Observe first that any PROP $(\PP,\oplus,0,\tau)$ induces a coloured operad $\EE_\PP$ whose colours are the objects of $\PP$ different from $0$, and whose operations are given by $\EE_\PP(n_1,\dots,n_k;n)=\PP(n_1\oplus\cdots\oplus n_k,n)$. Note that the unit of $\PP$ is not a colour in $\EE_\PP$; it contributes only in the definition of the constants $\EE_\PP(;n)=\PP(0,n)$. The substitution maps of $\EE_\PP$ are defined in an obvious manner. For a small category $\CC$, we define a coloured operad $\FF_\CC$ whose colours are the objects of $\CC$ and whose operations are given by $\FF_\CC(c_1,\dots,c_k;c)=\CC(c_1,c)\times\cdots\times\CC(c_k,c)$. Again, the substitution maps of $\FF_\CC$ are defined in an obvious manner. Both constructions are functorial. Moreover, for $\PP=\Delta\Sg_+$, we have a canonical map of coloured operads $\EE_{\Delta\Sg_+}\to\FF_{\Delta\Sg}$ induced by the inclusions $[n_i]\to[n_1]*\cdots *[n_k]$. By Proposition \ref{duallattice}, the lattice path operad may then be identified with the following pullback\begin{diagram}[small]\LL&\rTo&\EE_{\Delta\Sg_+}\\\dTo&&\dTo\\\FF_\Delta&\rTo&\FF_{\Delta\Sigma}\end{diagram}in the category of coloured operads.\end{rmk}

\subsection{Filtration by complexity}

For each $1\leq i< j\leq k$, there are canonical projection functors $p_{ij}:[n_1 +1]\otimes\cdots\otimes[n_k +1]\longrightarrow[n_i+1]\otimes[n_j+1].$ These functors, together with the unique functor in $\Cat_{*,*}([1],[n+1])$, induce maps
  $$\phi_{ij}:\LL(n_1,\ldots, n_k; n)\to\LL(n_i,n_j;0),\quad 1\leq i<j\leq k.$$

\begin{dfn}For each $x\in\LL(n_1,\dots,n_k;n)$ and each $1\leq i<j\leq k$, let $c_{i j}(x)$ be the number of changes of directions (i.e. corners) in the lattice path $\phi_{i j}(x)$. The \emph{complexity index} $c(x)$ of $x\in \LL(n_1,\dots,n_k;n)$ is defined by $$c(x)=\max_{1\leq i<j\leq k}c_{i j}(x).$$The \emph{$m$-th filtration stage} $\LL_m$ of the lattice path operad $\LL$ is defined by$$\LL_m(n_1,\dots,n_k;n)=\{x\in\LL(n_1,\dots,n_k;n)\,|\,c(x)\leq m\}.$$\end{dfn}

It is readily verified that $(\LL_m)_{m\geq 0}$ defines an exhaustive filtration of $\LL$ by $\NN$-coloured suboperads. The suboperad $\LL_0$ only contains unary operations and coincides with the underlying category $\LL_u$ of $\LL$. Therefore, for each $m\geq 0$, we have $(\LL_m)_u=\LL_m\cap\LL_u=\LL_u=\Delta$, cf. (\ref{Joyal}). In particular, each $\LL_m$-algebra $(X_n)_{n\geq 0}$ carries a canonical \emph{cosimplicial} structure. Moreover, for $m\geq 1$, the coloured operad $\LL_m$ contains the permutation operad as the endomorphism operad of colour $0$, i.e. $X_0$ has the structure of an associative monoid. With increasing $m$, an $\LL_m$-algebra structure on $(X_n)_{n\geq 0}$ introduces higher and higher commutativity constraints on the monoid $X_0$ in a compatible way with the cosimplicial structure on $(X_n)_{n\geq 0}$.

\begin{rmk}In the dual description (\ref{duallattice}) of the lattice path operad, the maps $\phi_{ij}:\LL(n_1,\dots,n_k;n)\to\LL(n_i,n_j;0)$ are induced by the inclusions $[n_i]*[n_j]\inc[n_1]*\cdots*[n_k]$ and the unique map $[n]\to[0]$. In particular, the individual complexity indices $c_{ij}(x)$ of a lattice path $x$ count the number of times the associated integer-string switches from $i$ to $j$ or from $j$ to $i$, if one runs through it from left to right, cf. \cite[Definition 8.1]{MS3} from where we borrowed the term complexity.\end{rmk}

\begin{rmk}\label{discrete}Since $\LL_m$ is discrete, $\LL_m$ acts in \emph{any} cocomplete, closed symmetric monoidal category $\EE$ in exactly the same way as e.g. a discrete group acts. Formally, we just use that $\Set\to\EE:S\mapsto\coprod_{S}I_\EE$ is a strong monoidal functor, since the tensor $-\otimes_\EE-$ commutes with coproducts in each variable. We tacitly apply this functor at each time $\LL_m$ is said to act on a sequence $(X_n)_{n\geq 0}$ of objects of $\EE$.\end{rmk}

\begin{rmk}\label{comparison}For $\EE=\Top$, the structure of an \emph{$\LL_m$-algebra} is strongly related to McClure and Smith's \cite{MS3} structure of a \emph{$\Xi^m$-algebra}. The only difference comes from the fact that for McClure and Smith the underlying objects are \emph{coaugmented} cosimplicial spaces, while for us the underlying objects are just cosimplicial spaces. See \cite[Remark 6.11]{MS3} how to get from their setting to ours. Beside this difference, the passage from an $\LL_m$-algebra to a $\Xi^m$-algebra is the passage from the coloured operad $\LL_m$ to the functor-operad $\xi(\LL_m)=\Xi^m$ as explained in Proposition \ref{convolution}. To be more precise, the colimit presentation of the functor-operads $\Xi$ (in \cite[page 1125]{MS3}) and $\Xi^m$ (in \cite[Definition 8.4]{MS3}) is \emph{isomorphic} to the coend-formula of the functor-operads $\xi(\LL)$ and $\xi(\LL_m)$, cf. the definition preceding (\ref{convolution}), since McClure and Smith's triples $\overline{k}\leftarrow T\rightarrow S$ are just another combinatorial description of lattice paths, equivalent to our integer-string representation (\ref{string}). The operad structure of $\LL$ is implicitly present in diagram (6.2) of \cite[page 1126]{MS3}. It was one of the purposes of this article to motivate somehow the constructions of McClure and Smith by means of the lattice path operad equipped with its complexity filtration. Moreover, it seems to us that the concept of an algebra over a coloured operad is conceptually easier than that of an algebra over a functor-operad. Modulo the differences just mentioned, Proposition \ref{L2} below is thus a reformulation of \cite[Propositions 2.3 and 10.3]{MS3}. Other relevant sources can be found in \cite{Ba1} and \cite{TT}.\end{rmk}

Cosimplicial objects in any cocomplete category $\EE$ carry a (non-symmetric) monoidal structure, induced from the monoidal structure on $\Delta_+$ by Day convolution \cite{Da}. This monoidal structure on cosimplicial objects will be denoted $\square$, cf. \cite[Proposition 2.2]{MS3}. Following Gerstenhaber-Voronov \cite{GV}, a non-symmetric operad $\OO$ will be called \emph{multiplicative} if $\OO$ comes equipped with a map of non-symmetric operads $\Ass\to\OO$ where $\Ass$ is the operad for monoids, i.e. $\Ass(n)=I_\EE$ for $n\geq 0$.

\begin{prp}\label{L2}Let $\EE$ be a cocomplete, closed symmetric monoidal category. The category of $\LL_1$-algebras (resp. $\LL_2$-algebras) in $\,\EE$ is isomorphic to the category of cosimplicial $\square$-monoids (resp. multiplicative non-symmetric operads) in $\EE$.\end{prp}

\begin{proof}The dual description (\ref{duallattice}) of $\LL$ implies a canonical splitting$$\LL_1(n_1,\dots,n_k;n)\cong\Delta_+([n_1]*\cdots*[n_k],[n])\times\Sg_k$$which identifies $\LL_1$ with the \emph{symmetrization} of the non-symmetric $\NN$-coloured operad $\Delta_+([-]*\cdots*[-];[-])$. Since symmetrization does not change the notion of algebra, $\LL_1$-algebras are the same as $\Delta_+([-]*\cdots*[-];[-])$-algebras. But the latter are (almost by definition) the same as cosimplicial $\square$-monoids.

We shall now identify $\LL_2$ with the $\NN$-coloured operad $\OO$ whose algebras are multiplicative non-symmetric operads. In \cite[1.5.6]{BM2} the coloured operad for symmetric operads is described in detail. Forgetting the symmetries and adding multiplications leads to the following description: $\OO(n_1,\dots,n_k;n)$ is the set of (isomorphism classes of) pairs $(T,v)$ consisting of a planar rooted tree $T$ with $n$ input edges, and a labelling $v:\{1,\dots,k\}\inc V(T)$ of $k$ vertices of $T$, such that the vertex $v(i)$ has $n_i$ input edges, vertices of valence $2$ are labelled, and any inner edge has at least one labelled extremity. (By convention, the external edges, i.e. the input and root edges of the tree, are open on one side; stumps, i.e. labelled or unlabelled vertices with no input edge, are allowed). The operad substitution maps of $\OO$ are defined by an obvious substitution of trees into vertices. We have to construct an isomorphism of $\NN$-coloured operads $\LL_2\cong\OO$.

The labelled tree $(T_x,v_x)$ associated to $x\in\LL_2(n_1,\dots,n_k;n)$ is constructed as follows: for each $i=1,\dots,k,$ the representing integer-string of $x$ has a minimal substring $\gamma_i(x)$ containing all occurrences of $i$. The complexity index of the lattice path is $\leq 2$ \emph{if and only if} these substrings either have empty intersection or contain each other, i.e. form a \emph{nested sequence} of integer-strings. The labelled tree $(T_x,v_x)$ is inductively defined by assigning to each substring $\gamma_i(x)$ a vertex $v(i)$ supporting the subtrees which arise from the substrings $\gamma_j(x)$ contained in $\gamma_i(x)$. If there is an unnested sequence of consecutive substrings, the latter are joined above a non-labelled vertex of $T_x$. To each subdivision of the integer-string representing $x$ is assigned an input edge of $T_x$; the location of this input edge is uniquely determined by the location of the subdivision. Conversely, given an element $(T,v)\in\OO(n_1,\dots,n_k;n)$, the lattice path $x\in\LL(n_1,\dots,n_k;n)$ such that $(T,v)=(T_x,v_x)$ is obtained by running through $T$ from left to right via the unique edge-path which begins and ends at the root-edge and which takes \emph{each} edge of $T$ \emph{exactly twice} (in opposite directions): the integer-string is given by the sequence of labelled vertices through which the edge-path goes; the subdivisions are given by the sequence of input edges through which the edge-path goes.

We leave it to the reader to check that these two assignments are mutually inverse, and that the resulting bijection respects the operad structures.\end{proof}

\begin{rmk}The three filtration stages $\LL_0$, $\LL_1$ and $\LL_2$ are the only one for which we are able to describe the algebras by generators and relations like above. It would be useful to have analogous descriptions for $m\geq 3$. The following proposition is a first indication of the close relationship between $\LL_m$-algebras and $m$-fold iterated loop spaces. It is a reformulation of \cite[Lemma 11.3]{MS3}. The case $m=2$ (together with Proposition \ref{L2}) recovers \cite[Theorem 3.3]{Si} of Sinha.\end{rmk}

\begin{prp}\label{coalgebra}The simplicial $m$-sphere $S^m=\Delta[m]/\partial\Delta[m]$ is an $\LL_m$-coalgebra in the category $(\Fin_*,\vee,*)$ of finite pointed sets.\end{prp}

\begin{proof}Each $x\in\LL(n_1,\ldots,n_k;n)$ induces a $k$-tuple of operators $(x_1,\dots,x_k)\in\Delta([n_1],[n])\times\cdots\times\Delta([n_k],[n])$, cf. (\ref{duallattice}), which coact on $y\in(S^m)_n$ by $x^*(y)=(x_1^*(y),\dots,x_k^*(y))\in(S^m)_{n_1}\times\cdots\times(S^m)_{n_k}$. We shall show that if the complexity index fulfills $c(x)\leq m$ then $x^*(y)\in(S^m)_{n_1}\vee\cdots\vee(S^m)_{n_k}$.

Assume that $x^*(y)$ does not belong to the wedge, i.e. at least two projections of $x^*(y)$ (say in directions $i$ and $j$) are not at the base point of the $m$-sphere $S^m$. Then $y$ is not at the base-point, thus $n\geq m$, and $x_i^*(y)$ and $x_j^*(y)$ are not at the base point, thus the $n$-fold subdivided integer-string representing $x$ contains occurences of $i$ and of $j$ in each of its $n+1$ substrings, whence $c(x)\geq n+1>m$.\end{proof}

We define the \emph{standard cosimplicial space} to be $\delta_{top}:\Delta\to\Top:[n]\mapsto\Delta_n$ where $\Delta_n$ is the standard euclidean $n$-simplex, and where the simplicial operators act as usually by affine extension of the vertex action.

\begin{cor}The $m$-fold loop space $\Omega^mY$ of a based topological space $Y$ carries a canonical action by the $\delta_{top}$-condensation $\Coend_{\LL_m}(\delta_{top})$ of $\LL_m$.\end{cor}

\begin{proof}By adjunction, the $m$-fold loop space $\Omega^mY=\uTop_*(|S^m|_{\delta_{top}},Y)$ is homeomorphic to the $\delta_{top}$-totalization of the cosimplicial space $(Y,*)^{(S^m,*)}$. The latter is an $\LL_m$-algebra by (\ref{coalgebra}), whence by (\ref{condensation}) a $\Coend_{\LL_m}(\delta_{top})$-action on $\Omega^mY$.\end{proof}

A similar argument applies to higher Hochschild cohomology of commutative algebras in the sense of Pirashvili \cite{Pi}. Indeed, using that $(\Fin_*,\vee,*)$ is the PROP for commutative monoids, any unital commutative ring $A$ gives rise to a strong symmetric monoidal functor $A^{\otimes -}:(\Fin_*,\vee,*)\to(\Mod_\ZZ,\otimes,\ZZ)$; therefore, composing the $\LL_m$-coalgebra $S^m$ with $A^{\otimes -}$ yields an $\LL_m$-coalgebra in $\Mod_\ZZ$ which we shall abbreviate by $A^{\otimes \binom{-}{m}}$, since Joyal-duality (\ref{duality}) induces a canonical bijection between $(S^{m})_n-\{*\}$ and the set ${\binom{n}{m}}$ of subsets of cardinality $m$ of $\{1,\dots,n\}$.

Totalization of $\Hom_\ZZ(A^{\otimes\binom{-}{m}},A)$ with respect to the \emph{standard cosimplicial chain complex} $\delta_\ZZ:\Delta\overset{\delta_{yon}}{\to}\Set^{\Delta^\op}\overset{N_*}{\to}\Ch(\ZZ)$ yields a cochain complex $\HCC_{(m)}^*(A;A)$ whose cohomology computes Pirashvili's higher Hochschild cohomology $\HH_{(m)}^*(A;A)$. In particular, $\HCC_{(1)}^*(A;A)$ is the usual normalized Hochschild cochain complex of $A$.

\begin{cor}The higher Hochschild cochains $\HCC_{(m)}^*(A;A)$ of a commutative algebra $A$ carry an action by the $\delta_\ZZ$-condensation $\Coend_{\LL_m}(\delta_\ZZ)$ of $\LL_m$.\end{cor}

\begin{proof}The simplicial module $A^{\otimes\binom{-}{m}}$ is an $\LL_m$-coalgebra, thus the cosimplicial module $\Hom_\ZZ(A^{\otimes\binom{-}{m}},A)$ is an $\LL_m$-algebra, whence by (\ref{condensation}) a canonical $\Coend_{\LL_m}(\delta_\ZZ)$ operad action on $\HCC_{(m)}^*(A;A)$.\end{proof}

\begin{rmk}McClure and Smith show that $\Coend_{\LL_m}(\delta_{top})$ is a topological $E_m$-operad, cf. \cite[Theorem 9.1]{MS3} and (\ref{mainexamples}a). It follows from (\ref{mainexamples}c) below that $\Coend_{\LL_m}(\delta_\ZZ)$ is an $E_m$-operad in chain complexes. Observe however, while the $E_m$-action on $\Omega^mY$ is ``optimal'', the $E_m$-action on $\HCC^*_{(m)}(A;A)$ is not. Indeed, the case $m=1$ suggests that $\HCC^*_{(m)}(A;A)$ carries a canonical $E_{m+1}$-action. A first step in this direction has been made by Ginot, cf. \cite[Theorem 3.4]{Gi}, where an $H_*(E_{m+1})$-action on $\HH_{(m)}^*(A;A)$ is constructed. This is related to higher forms of the Deligne conjecture as formulated by Kontsevich.\end{rmk}\vspace{1ex}

The $\LL_m$-coalgebra structure (\ref{coalgebra}) on the simplicial $m$-sphere $S^m$ derives from a general $\LL$-coalgebra structure on simplicial sets considered as graded objects in $(\Set,\times)$. By abelianization, dualization and condensation, this set-theoretical $\LL$-coalgebra structure induces a canonical $E_\infty$-structure on cochains:

\begin{prp}For each simplicial set $X$, the simplicial abelian group $\ZZ[X]$ (resp. cosimplicial abelian group $\Hom(\ZZ[X],\ZZ)$) is an $\LL$-coalgebra (resp. $\LL$-algebra) in abelian groups. In particular, the normalized cochain complex $N^*(X;\ZZ)$ carries a canonical $E_\infty$-action by $\Coend_\LL(\delta_\ZZ)$.\end{prp}

\begin{proof}The $\LL$-coaction $\ZZ[\LL(n_1,\dots,n_k;n)]\otimes\ZZ[X_n]\to\ZZ[X_{n_1}]\otimes\cdots\otimes\ZZ[X_{n_k}]$ is defined on generators by $x\otimes y\mapsto x_1^*(y)\otimes\cdots\otimes x_k^*(y)$ where $y\in X_n$ and $(x_1,\dots,x_k)$ are the components of $x\in\LL(n_1,\dots,n_k;n)$. The $\LL$-action derives by dualization. The second statement follows from (\ref{condensation}) and the fact that $\delta_\ZZ$-totalization of a cosimplicial abelian group may be identified with its conormalization, cf. (\ref{conorm}).\end{proof}

\begin{rmk}In (\ref{mainexamples}c) we show that the $\delta_\ZZ$-condensation $\Coend_\LL(\delta_\ZZ)$ contains the \emph{surjection operad} $\XX$ of \cite{BF,MS2} as a \emph{suboperad}. The $\Coend_\LL(\delta_\ZZ)$-action on the cochains $N^*(X;\ZZ)$ restricts therefore to a $\XX$-action, and it can be checked that the resulting action coincides (up to a sign) with the one described in \cite{BF,MS2}. Moreover, the inclusion $\XX\inc\Coend_\LL(\delta_\ZZ)$ is a weak equivalence of chain operads. Therefore, \cite[Lemma 1.6.1]{BF} implies that $\XX$ and $\Coend_\LL(\delta_\ZZ)$ are $E_\infty$-chain operads. The existence of an explicit $E_\infty$-action on $N^*(X;\ZZ)$ is important since by theorems of Mandell \cite{Man}, such an action determines (under suitable finiteness and completeness assumptions) the weak homotopy type of $X$.\end{rmk}

\section{\label{complexity} Condensation of the lattice path operad}

In order to analyse the operads obtained from the lattice path operad through $\delta$-condensation (\ref{condensation}) we need to recall some constructions from \cite{Be2} concerning the so-called \emph{complete graph operad} $\cgr.$

A partially ordered set (for short: poset) $\Aa$ will be identified with the category having same objects as $\Aa$, and having a morphism $\al\to\be$ whenever $\al\leq\be$ in $\Aa$. The \emph{classifying space} $B\Aa$ of $\Aa$ is the geometric realization of the nerve of $\Aa$.

\subsection{Cellulation of topological operads by poset-operads}\label{cellulation}Given a topological space $X$ and a poset $\Aa$ we say that $X$ admits an \emph{$\Aa$-cellulation} if there is a functor $c:\Aa\to\Top:\al\mapsto c_\al$ fulfilling the following three properties:\begin{enumerate}\item[(i)]$\colim_{\al\in\Aa}c_\al\cong X$;\item[(ii)] for each $\al\in\Aa$, the canonical map $\colim_{\be<\al}c_\be\to c_\al$ is a closed cofibration;\item[(iii)] for each $\al\in\Aa$, $c_\al$ is contractible.\end{enumerate}

Fixing the isomorphism (i) allows us to consider $c_\al$ as a closed subset of $X$ which we call the \emph{cell} labelled by $\al$. By (ii) the colimit $\partial c_\al=\colim_{\be<\al}c_\be$ is a closed subset of $c_\al$ which we call the \emph{boundary} of $c_\al$. Points in $c_\al-\partial c_\al$ are called \emph{interior points} of $c_\al$. A cell $c_\al$ is called \emph{proper} if it has interior points or, equivalently, if the closed cofibration $\partial c_\al\to c_\al$ is not an isomorphism.

Any space $X$ with $\Aa$-cellulation induces homotopy equivalences\begin{diagram}X&\lTo^{\sim}&\hocolim_{\al\in\Aa}c_\al&\rTo^{\sim}&B\Aa.\end{diagram}The left arrow is a homotopy equivalence since condition (ii) expresses that the functor $\al\mapsto c_\al$ is \emph{Reedy-cofibrant} with respect to \emph{Str\o m's model structure} on $\Top$ so that colimit and homotopy colimit are homotopy equivalent; the right arrow is a homotopy equivalence by (iii) since the classifying space of the poset $\Aa$ is the homotopy colimit of the constant diagram, i.e. $B\Aa=\hocolim_{\al\in\Aa}*$.

Given a topological operad  $X=(X(k))_{k\geq 0}$ and a poset operad $\Aa=(\Aa(k))_{k\geq 0}$, we say that $X$ admits an \emph{$\Aa$-cellulation} if, for each $k\geq 0$, $X(k)$ admits a $\Sg_k$-equivariant $\Aa(k)$-cellulation such that the operad substitution maps of $X$ take the cell-product $c_\al\times c_{\al_1}\times\cdots\times c_{\al_k}\subset X(k)\times X(n_1)\times\cdots\times X(n_k)$ into the cell $c_{\al(\al_1,\dots,\al_k)}\subset X(n_1+\cdots+n_k).$ Since homotopy colimits commute with products we get diagrams of homotopy equivalences\begin{diagram}X(k)&\lTo^{\sim}&\hocolim_{\al\in\Aa(k)}c_\al&\rTo^{\sim}&B\Aa(k)\end{diagram}in which all three objects define topological operads and both arrows are morphisms of topological operads. Therefore, \emph{if the operad $X$ admits a cellulation by the poset-operad $\Aa$ then $X$ and $B\Aa$ are weakly equivalent as topological operads}.

\subsection{The complete graph operad}For each $k\geq 0$, let $\cgr(k)=\NN^{\binom{k}{2}}\times\Sg_k$ and define an inclusion $\rho_k:\cgr(k)\inc\prod_{1\leq i<j\leq k}\cgr(2)$ by $(\mu,\sg)\mapsto((\mu_{ij},\sg_{ij}))_{1\leq i<j\leq k}$ where $\mu_{ij}$ is the $ij$-th component of $\mu$ and $\sg_{ij}\in\Sg_2$ is neutral (resp. non-neutral) if $\sg(i)<\sg(j)$ (resp. $\sg(i)>\sg(j)$).

The set $\cgr(2)$ is partially ordered by $(m,\sg)\leq(n,\tau)$ iff $m<n$ or $(m,\sg)=(n,\tau)$; the partial order of $\cgr(k)$ is induced from the product-order of the target of $\rho_k$.

An operad substitution map $\cgr(k)\times\cgr(n_1)\times\cdots\times\cgr(n_k)\to\cgr(n_1+\cdots+n_k)$ is defined by $((\mu,\sg);(\mu_1,\sg_1),\dots,(\mu_k,\sg_k))\mapsto(\mu(\mu_1,\dots,\mu_k),\sg(\sg_1,\dots,\sg_k))$ where the right component is given by the permutation operad, and the left component by$$\mu(\mu_1,\dots,\mu_k)_{ij}=\begin{cases}\mu_{ij}&\text{if }i,j\text{ belong to different blocks of }\binom{n_1+\cdots+n_k}{2},\\(\mu_s)_{ij}&\text{if  }i,j\text{ belong to the same block }\binom{n_s}{2}\text{ of }\binom{n_1+\cdots+n_k}{2}.\end{cases}$$

The poset-operad $\cgr$ is the \emph{complete graph operad} from \cite{Be2}, cf. also \cite{BFSV,BFV,MS3}. An element $(\mu,\sg)\in\cgr(k)$ can be considered as a complete graph on the vertex set $\{1,\dots,k\}$ equipped with an \emph{edge-labeling} by the integers $\mu_{ij}$ and an \emph{acyclic edge-orientation} $\sg(1)\to\sg(2)\to\cdots\to\sg(k)$. For consistency, the symmetric group $\Sg_k$ acts here on $\cgr(k)$ from the left, i.e. $\tau_*(\mu,\sg)=(\tau\mu,\tau\sg)$ where $(\tau\mu)_{ij}=\mu_{\tau^{-1}(i)\tau^{-1}(j)}$. The operad substitution maps correspond to a substitution of complete graphs into vertices of a complete graph. Indeed, there is a definition of the orientations $\sg(\sg_1,\dots,\sg_k)_{ij}$ completely analogous to the definition of $\mu(\mu_1,\dots,\mu_k)_{ij}$ above.

Below we will use an extension $\cgr^{ex}=(\cgr^{ex}(k))_{k\geq 0}$ of $\cgr$ consisting of those labelled and oriented complete graphs $((\mu_{ij},\sg_{ij}))\in\prod_{1\leq i<j\leq k}\cgr(2)$ which contain no oriented cycle $i_0\overset{\mu_{i_0,i_1}}{\lrto}i_1\overset{\mu_{i_1,i_2}}{\lrto}i_2\overset{\mu_{i_2,i_3}}{\lrto}\cdots\overset{\mu_{i_{s-2},i_{s-1}}}{\lrto}i_{s-1}\overset{\mu_{i_{s-1},i_0}}{\to}i_s=i_0$ such that $\mu_{i_0,i_1}=\mu_{i_1,i_2}=\cdots=\mu_{i_{s-1},i_0}$. The operad substitution maps of $\cgr$ extend in an obvious way to $\cgr^{ex}$. This \emph{extended complete graph operad} has been introduced by Brun, Fiedorowicz and Vogt in \cite{BFV}. It is remarkable (cf. \cite{Ba2}) that an element $(\mu,\sg)\in\cgr^{ex}(k)$ determines, and is determined by, a sequence of \emph{complementary relations} on $\{1,\dots,k\}$, a notion introduced by Kontsevich and Soibelman in \cite{KS1}. The (extended) complete graph operad is filtered by $\cgr_m$ (resp. $\cgr_m^{ex}$) where $(\mu,\sg)\in\cgr_m$ (resp. $\cgr_m^{ex}$) iff $\mu_{ij}<m$ for all $i<j$.

\begin{thm}[\cite{BFSV,Be2,BFV}]\label{Em} The operad of little $m$-cubes has a $\cgr_m^{ex}$-cellulation. The inclusion $\cgr_m\inc\cgr_m^{ex}$ is a weak equivalence of poset-operads. In particular, the classifying operads $B\cgr_m$ and $B\cgr_m^{ex}$ are topological $E_m$-operads.\end{thm}

\begin{prp}\label{totcom}The individual complexity indices of lattice paths determine a morphism of coloured operads $c_{tot}:\LL\to\cgr$ which preserves the filtrations (and which takes all colours of $\LL$ to the single colour of $\,\cgr$).\end{prp}

\begin{proof}The individual complexity indices determine the left components of a map
$$c_{tot}:\LL(n_1,\ldots, n_k; n) \rightarrow \cgr(k):x\mapsto(c_{ij}(x)-1,\sg_{ij}(x))_{1\leq i<j\leq k}$$
where $\sg_{ij}(x)\in\Sg_2$ is neutral (resp. non-neutral) if the first step of $x$ in direction $i$ precedes (resp. does not precede) the first step of $x$ in direction $j$. In other words, the $\sg_{ij}(x)$ are the components of a permutation $\sg_x\in\Sg_k$ which is determined by the property that the sequence $\sg_x(1)\sg_x(2)\cdots\sg_x(k)$ is the first subsequence (w/to the lexicographical order) of the integer-string representation of $x$, which represents a permutation.

That $\LL\to\cgr$ respects the filtrations follows immediately from the definitions. Projection on the second factor induces an operad map $\cgr\to\Sg$ and we have already observed in (\ref{string}) that the composite $\LL\to\cgr\to\Sg$ is a map of operads. Therefore it remains to be shown that the individual complexity indices $c_{ij}(x(x_1,\dots,x_k))$ of a composite lattice path $x(x_1,\dots,x_k)$ can be identified with the corresponding individual labelings of $c_{ij}(x)(c_{ij}(x_1),\dots,c_{ij}(x_k))$ in the complete graph operad. This follows again by direct inspection.\end{proof}

\subsection{Standard systems of simplices in monoidal model categories}\label{sss}In order to relate the \emph{$\delta$-condensation} of $\LL_m$ to \emph{$E_m$-operads}, we need to specify some conditions on $\EE$, $\delta$ and on their interplay with $\LL$. Although we give a general framework, our main interest are the concrete cases $\EE=\Top$, $\EE=\Set^{\Delta^\op}$ and $\EE=\Ch(\ZZ)$, with the standard cosimplicial objects $\delta_{top}$, $\delta_{yon}$, $\delta_\ZZ$ as defined at the end of Section \ref{latticepath}. In general, we shall assume that $\EE$ is a \emph{monoidal model category} in Hovey's sense \cite{Ho}, i.e. at once a closed symmetric monoidal category and a Quillen model category such that so-called unit and pushout-product axioms hold. This is the case for the Quillen model structure on $\Top$ resp. $\Set^{\Delta^\op}$, and for the projective model structure on $\Ch(\ZZ)$. Moreover, we shall assume that $\delta$ is a \emph{standard system of simplices} for $\EE$ in the sense of \cite[Definition A.6]{BM}. This means that\begin{itemize}\item[(i)]$\delta^\cdot$ is cofibrant for the Reedy model structure on $\EE^\Delta$;\item[(ii)]$\delta^0$ is the unit $I_\EE$ of $\EE$, and the simplicial operators $[m]\to[n]$ act as weak equivalences $\delta^m\to\delta^n$ in $\EE$;\item[(iii)]the realization functor $|\!-\!|_\delta=(-)\otimes_\Delta\delta^\cdot:\EE^{\Delta^\op}\to\EE$ is a symmetric monoidal functor whose structural maps $|X|_\delta\otimes_\EE|Y|_\delta\to|X\otimes_\EE Y|_{\delta}$ are weak equivalences for Reedy-cofibrant objects $X,Y$ of $\EE^{\Delta^\op}$.\end{itemize}The three standard cosimplicial objects $\delta_{top},\,\delta_{yon},\,\delta_\ZZ$ are such standard systems of simplices, cf. \cite[A.13 and A.16]{BM}. In general, a monoidal model category $\EE$ with a standard system of simplices $\delta$ admits functorial \emph{cosimplicial framings} $(-)\otimes_\EE\delta^\cdot$, and therefore a canonical way of constructing \emph{homotopy colimits}, cf. Hovey \cite{Ho}. These homotopy colimits are compatible with the tensor $\otimes_\EE$ in the same way as homotopy colimits of topological spaces (resp. simplicial sets) are compatible with cartesian product. Moreover, since the composite functor $\Set^{\Delta^\op}\to\EE^{\Delta^\op}\to\EE$ is a symmetric monoidal left Quillen functor \cite[A.13]{BM}, there is a well defined (i.e. homotopy invariant) concept of \emph{$E_m$-operad} in $\EE$, namely, any symmetric operad in $\EE$ that is weakly equivalent to the $\delta$-realization of a simplicial $E_m$-operad.

\subsection{$\delta$-realization and $\delta$-totalization}Recall from (\ref{condensation}) that the underlying collection of the $\delta$-condensation of $\LL$ is given by$$\Coend_\LL(\delta)(k)=\uHom_\Delta(\delta^\cdot,\LL(\overbrace{-,\cdots,-}^k;\cdot)\otimes_{\Delta\times\cdots\times\Delta}\overbrace{\delta^-\otimes_\EE\cdots\otimes_\EE\delta^-}^k),\,k\geq 0.$$It is thus obtained as the composite of a (multisimplicial) $\delta$-\emph{realization} followed by a (cosimplicial) $\delta$-\emph{totalization}. In order to simplify notation, we shall abbreviate $$\Coend_\LL(\delta)(k)=\uHom_\Delta(\delta^\cdot,|\LL(\overbrace{-,\cdots,-}^k;\cdot)|_{\delta^{\otimes k}})=\uHom_\Delta(\delta^\cdot,\xi_k(\delta)^\cdot).$$In particular, $\xi_k(\delta)^\cdot=|\LL(-,\cdots,-;\cdot)|_{\delta^{\otimes k}}$ denotes a cosimplicial object of $\EE$ which comes equipped with a canonical $\Sg_k$-action. The resulting \emph{cosimplicial collection} carries a $\Coend_\LL(\delta)$-\emph{right module} structure, i.e. there are cosimplicial maps$$\xi_k(\delta)^\cdot\otimes_\EE\Coend_\LL(\delta)(n_1)\otimes_\EE\cdots\otimes_\EE\Coend_\LL(\delta)(n_k)\to\xi_{n_1+\cdots+n_k}(\delta)^\cdot$$from which  the right action of $\Coend_\LL(\delta)$ on itself can be deduced through adjunction and $\delta$-totalization. The cosimplicial object $\xi_k(\delta)^\cdot$ is the colimit of a functor$$\cgr(k)\to\EE^\Delta:(\mu,\sg)\mapsto\xi_{(\mu,\sg)}(\delta)^\cdot$$where the cosimplicial object $\xi_{(\mu,\sg)}(\delta)^\cdot$ denotes the $\delta$-realization of those lattice paths $x\in\LL$ whose total complexity index fulfills $c_{tot}(x)\leq(\mu,\sg)$, cf. (\ref{totcom}), i.e.$$\xi_{(\mu,\sg)}(\delta)^\cdot=|\{x\in\LL(\overbrace{-,\cdots,-}^k;\cdot),\,c_{tot}(x)\leq(\mu,\sg)\}|_{\delta^{\otimes k}}.$$\vspace{1ex}

We introduce the following terminology: a \emph{weak equivalence} in $\EE$ is called \emph{universal} if any pullback of it is again a weak equivalence. For instance, any trivial fibration is a universal weak equivalence.  Any map $f:Y\to X$ admitting a section $i:X\to Y$ such that $if:Y\to Y$ is fiber homotopy equivalent to $id_Y$ is a universal weak equivalence. More generally, any map $f:Y\to X$, which is fiber homotopy equivalent to a trivial fibration $g:Y\to X$, is a universal weak equivalence.

\begin{dfn}\label{reductive}Let $\delta$ be a standard system of simplices in $\EE$. The lattice path operad $\LL$ is called \emph{$\delta$-reductive} if for each $n\geq 0$ and $k\geq 0$, the map $\xi_k(\delta)^n\to\xi_k(\delta)^0$ is a universal weak equivalence in $\EE$; $\LL$ is called \emph{strongly $\delta$-reductive} if in addition the induced map $\Coend_\LL(\delta)(k)\to\xi_k(\delta)^0$ is also a universal weak equivalence in $\EE$.\end{dfn}

One way to obtain strong $\delta$-reductivity is to show that the cosimplicial map $\xi(\delta)^\cdot\to\xi(\delta)^0$ is a trivial Reedy fibration with respect to the Reedy model structure on $\EE^\Delta$, where $\xi(\delta)^0$ is considered as a constant cosimplicial object. Indeed, a trivial Reedy fibration is objectwise a trivial fibration, which yields $\delta$-reductivity. Moreover, by (\ref{sss}i), $\delta$-totalization is a right Quillen functor $\EE^\Delta\to\EE$, which implies that the induced map $\Coend_\LL(\delta)(k)\to\xi_k(\delta)^0$ is a trivial fibration in $\EE$. We shall see in (\ref{mainexamples}) below that the lattice path operad is strongly $\delta_{top}$-reductive as well as strongly $\delta_\ZZ$-reductive; it is $\delta_{yon}$-reductive, but \emph{not} strongly $\delta_{yon}$-reductive.

\begin{thm}\label{main}Let $\delta$ be a standard system of simplices in a monoidal model category $\EE$. If the lattice path operad $\LL$ is strongly $\delta$-reductive, then $\delta$-condensation of the different filtration stages $\LL_m$ of $\LL$ yields $E_m$-operads $\Coend_{\LL_m}(\delta)$ in $\EE$.\end{thm}

\begin{proof}We follow closely McClure-Smith's proof of \cite[Lemma 14.7]{MS3} which deals with the special case $\delta=\delta_{top}$. It will be sufficient to establish a zigzag of weak equivalences of operads in $\EE$:\begin{diagram}[small]\Coend_{\LL_m}(\delta)&\lTo^\sim&\Coend_{\widehat{\LL}_m}(\delta)&\rTo^\sim B_\delta\cgr_m\end{diagram}where $B_\delta:\Cat\lrto\Set^{\Delta^\op}\overset{|\!-\!|_\delta}{\lrto}\EE$ is the nerve functor followed by $\delta$-realization. Indeed, since the nerve of $\cgr_m$ is a simplicial $E_m$-operad by (\ref{Em}), and since $\delta$-realization is symmetric monoidal by (\ref{sss}iii), the classifying object $B_\delta\cgr_m$ is an $E_m$-operad so that the zigzag above shows that $\Coend_{\LL_m}(\delta)$ is an $E_m$-operad too. The intermediate operad $\Coend_{\widehat{\LL}_m}(\delta)$ is defined using the homotopy colimit idea of (\ref{cellulation}). More precisely, Proposition \ref{totcom} implies that the lattice path operad $\LL$ is a colimit of coloured subcollections $\LL_{(\mu,\sg)}$. Define$$\widehat{\LL}_m(n_1,\dots,n_k;n)=\hocolim_{(\mu,\sg)\in\cgr_m(k)}\LL_{(\mu,\sg)}(n_1,\dots,n_k;n),$$the homotopy colimit being taken inside $\EE$, cf. Remark \ref{discrete}. Since $c_{tot}:\LL\to\cgr$ is a map of coloured operads, and since by (\ref{sss}iii) homotopy colimits are compatible with the symmetric monoidal structure of $\EE$, the coloured collection $\widehat{\LL}_m$ is actually an $\NN$-coloured operad in $\EE$ with $\widehat{\LL}_u=\LL_u=\Delta$. Moreover, the corresponding cosimplicial object $\widehat{\xi}_k(\delta)^\cdot$ may be identified with $\hocolim_{(\mu,\sg)\in\cgr_m(k)}\xi_{(\mu,\sg)}(\delta)^\cdot$.

The left arrow in the zigzag above is induced by the canonical map of coloured operads $\widehat{\LL}_m\to\LL_m$. The right arrow in the zigzag above uses the identification $B_\delta\cgr_m=\Coend_{\widehat{\LL}_m}(I)$ where $I$ is the constant cosimplicial object with $I^n=I_\EE=\delta^0$ for each $n\geq 0$. The right arrow is thus induced by the cosimplicial map $\delta^\cdot\to I^\cdot$.

Since the complexity indices are independent of the cosimplicial structure, we get for each $n\geq 0$ and each $(\mu,\sg)\in\cgr(k)$ a pullback square\begin{diagram}[small]\xi_{(\mu,\sg)}(\delta)^n&\rTo&\xi_k(\delta)^n\\\dTo&&\dTo\\\xi_{(\mu,\sg)}(\delta)^0&\rTo&\xi_k(\delta)^0\end{diagram}in which the vertical maps are universal weak equivalences by $\delta$-reductivity of $\LL$. Since $\LL$ is \emph{strongly} $\delta$-reductive, the operad map $\Coend_{\widehat{\LL}_m}(\delta)\to\Coend_{\LL_m}(\delta)$ will be a weak equivalence as soon as the canonical map$$\hocolim_{\cgr_m(k)}\xi_{(\mu,\sg)}(\delta)^0\to\colim_{\cgr_m(k)}\xi_{(\mu,\sg)}(\delta)^0$$ is a weak equivalence for each $k\geq 0$. The latter follows from the fact that the functor $\cgr_m(k)\to\EE:(\mu,\sg)\mapsto\xi_{(\mu,\sg)}(\delta)^0$ is Reedy-cofibrant in virtue of (\ref{sss}i) (which implies that $|\!-\!|_\delta:\Set^{\Delta^\op}\to\EE$ preserves cofibrations, cf. \cite[A.11]{BM}). Similarly, in order to show that $\Coend_{\widehat{\LL}_m}(\delta)\to\Coend_{\widehat{\LL}_m}(I)$ is a weak equivalence, it is sufficient to show that the cosimplicial map $\delta^\cdot\to I^\cdot$ induces a weak equivalence$$\hocolim_{\cgr_m(k)}\xi_{(\mu,\sg)}(\delta)^n\to\hocolim_{\cgr_m(k)}\xi_{(\mu,\sg)}(I)^n$$for each $k\geq 0$ and $n\geq 0$. Since $\xi_{(\mu,\sg)}(I)^n=I_\EE$, this amounts to showing that $\delta^n\to\delta^0$ induces a weak equivalence $\xi_{(\mu,\sg)}(\delta)^n\to I_\EE$ for each $n\geq 0$ and $(\mu,\sg)\in\cgr_m(k)$. This in turn follows from the pullback square above, from (\ref{sss}ii) (which implies that $|\!-\!|_\delta:\Set^{\Delta^\op}\to\EE$ preserves weak equivalences, cf. \cite[A.13]{BM}) and from Lemma \ref{Barratt-Eccles}, since the cell $\xi_{(\mu,\sg)}(\delta)^0$ is weakly equivalent to $|\xi_{(\mu,\sg)}(\delta_{yon})^0|_\delta$ by (\ref{sss}iii).\end{proof}

\begin{lem}\label{Barratt-Eccles}The cells $\xi_{(\mu,\sg)}(\delta_{yon})^0$ of the simplicial set $\xi_k(\delta_{yon})^0$ are weakly contractible for all $(\mu,\sg)\in\cgr(k)$ and all $k\geq 0$.\end{lem}

\begin{proof}Since $|\xi_{(\mu,\sg)}(\delta_{yon})^0|_{\delta_{top}}\cong\xi_{(\mu,\sg)}(\delta_{top})^0$, this follows from \cite[Lemma 14.8]{MS3}, where it is shown that $\xi_{(\mu,\sg)}(\delta_{top})^0$ is contractible.\end{proof}

\begin{exms}\label{mainexamples}We shall briefly discuss the three main examples $\delta_{top},\delta_{yon},\delta_{\ZZ}$.\vspace{1ex}

(a) \emph{The topological condensation of $\LL$}.\vspace{1ex}

The lattice path operad $\LL$ is \emph{$\delta_{top}$-reductive} by \cite[Proposition 12.7]{MS3} where it is shown that there exists a canonical homeomorphism $\xi_k(\delta_{top})^\cdot\cong\xi_k(\delta_{top})^0\times\delta_{top}^\cdot$ compatible with the projections to $\xi_k(\delta_{top})^0$. It follows that $\xi_k(\delta_{top})^n\to\xi_k(\delta_{top})^0$ is a universal weak equivalence for each $n\geq 0$ and $k\geq 0$.  Moreover, there is an induced homeomorphism $\Coend_\LL(\delta_{top})(k)\cong\xi_k(\delta_{top})^0\times\uHom_\Delta(\delta_{top},\delta_{top})$, which implies that $\LL$ is \emph{strongly $\delta_{top}$-reductive}. Theorem \ref{main} thus recovers \cite[Theorem 9.1]{MS3} that $\Coend_{\LL_m}(\delta_{top})$ is a topological $E_m$-operad, cf. Remark \ref{comparison}.

The contractible factor $\uHom_\Delta(\delta_{top},\delta_{top})$ enters in an essential way in the operad structure of $\Coend_\LL(\delta_{top})$. Salvatore observes in \cite{Sa} that $\uHom_\Delta(\delta_{top},\delta_{top})$ is homeomorphic to the space of \emph{order- and endpoint-preserving continous self-maps} of the unit-interval. The appendix of \cite{Sa} makes explicit the isomorphism between $\Coend_{\LL_2}(\delta_{top})$ and McClure-Smith's \cite{MS3} original construction, as outlined in Remark \ref{comparison} above. By the aforementioned splitting, suitable contractible submonoids of $\uHom_\Delta(\delta_{top},\delta_{top})$ yield $E_2$-suboperads of $\Coend_{\LL_2}(\delta_{top})$. In particular, a variant of Kaufmann's \emph{spineless cacti operad} can be obtained in this way, cf. \cite{Ka1,Sa,V2}.\vspace{1ex}

(b) \emph{The simplicial condensation of $\LL$.}\vspace{1ex}

Since $|\xi_k(\delta_{yon})^\cdot|_{\delta_{top}}\cong\xi_k(\delta_{top})^\cdot$ and since $\delta_{top}$-realization preserves pullbacks and reflects weak equivalences, $\delta_{top}$-reductivity implies $\delta_{yon}$-reductivity of $\LL$.  However, $\LL$ is \emph{not} strongly $\delta_{yon}$-reductive so that Theorem \ref{main} does \emph{not} apply in the simplicial setting. This can be seen by observing that $\Coend_{\LL_1}(\delta_{yon})$ is the initial simplicial operad (instead of being a simplicial $E_1$-operad). Indeed, the $\delta_{yon}$-realization of $\Delta([-]*\cdots*[-],[n])$ is isomorphic to $k$-fold iterated \emph{edgewise subdivision} of $\Delta[n]$, and $\delta_{yon}$-totalization of the latter is empty for $k\geq 2$. Despite of this negative result, there is an interesting map of simplicial operads from $\Coend_\LL(\delta_{yon})$ to Barratt-Eccles' $E_\infty$-operad $E\Sg=(E\Sg_k)_{k\geq 0}$, constructed as follows.

The simplicial set $\xi_k(\delta_{yon})^n$ is the simplicial realization of the $k$-simplicial set $\LL(-,\cdots,-;n)$. The simplicial realization of a multisimplicial set is isomorphic to the diagonal of the multisimplicial set. Therefore, the $d$-simplices of $\xi_k(\delta_{yon})^n$ may be identified with the elements of $\LL(d,\dots,d;n)$. According to (\ref{string}), such an element $x\in\LL(d,\dots,d;n)$ is represented by an $n$-fold subdivided integer-string containing $d+1$ times each integer $1,2,3,\dots,k$; the simplicial face (resp. degeneracy) operators act by omitting (resp. duplicating) the relevant integers.  Now, define a simplicial map $\xi_k(\delta_{yon})^0\to E\Sg_k$ by sending an integer-string $x$ containing $d+1$ times each integer $1,2,\dots,k$, to the unique $(d+1)$-tuple of permutations $(\sg_0,\dots,\sg_d)\in (E\Sg_k)_d$ with the property that the string $\sg_0(1)\sg_0(2)\cdots\sg_0(k)$ is the subsequence of $x$ consisting of the first occurences of each integer, $\sg_1(1)\sg_1(2)\cdots\sg_1(k)$ the subsequence of $x$ consisting of the second occurences of each integer, and so on. This simplicial map splits, i.e. $\xi_k(\delta_{yon})^0\cong E\Sg_k\times V_k\to E\Sg_k$. The fiber $V_k$ has as $d$-simplices all integer-strings containing $d+1$ times each integer $1,2,\dots,k$ in such a way that in each left-sided substring the number of 1's excedes the number of 2's which excedes the number of 3's, and so on (i.e., the associated lattice path never crosses the diagonal). The subcollection $V_k$ of $\xi_k(\delta_{yon})^0$ extends in an obvious way to a cosimplicial subcollection $V_k^\cdot$ of $\xi_k(\delta_{yon})^\cdot$. The cosimplicial collection $V^\cdot_k$ comes actually from a non-symmetric coloured suboperad of $\LL$. Therefore, totalization of $V^\cdot_k$ yields a \emph{non-symmetric} simplicial operad, and the splitting above induces a global decomposition $\Coend_\LL(\delta_{yon})\cong E\Sg\times \uHom_\Delta(\delta^\cdot_{yon},V^\cdot)$ compatible with the operad structures. We conjecture the second factor to be a simplicial $A_\infty$-operad, so that projection on the first factor defines a weak equivalence between $\Coend_\LL(\delta_{yon})$ and Barratt-Eccles' $E_\infty$-operad $E\Sg$. If the conjecture holds, the Smith-filtration of $E\Sg$ (cf. \cite{Be2}) would induce a filtration of $\Coend_\LL(\delta_{yon})$ by simplicial $E_m$-suboperads which is different from the complexity filtration, cf. \cite[Remark 8.2]{MS3}.\vspace{1ex}

(c) \emph{The chain complex condensation of $\LL$.}\vspace{1ex}

Recall that the standard cosimplicial chain complex $\delta_\ZZ:\Delta\to\Ch(\ZZ)$ is defined by composing the Yoneda-embedding $\delta_{yon}:\Delta\to\Set^{\Delta^\op}$ with the \emph{normalized} chain functor $N_*(-;\ZZ):\Set^{\Delta^\op}\to\Ch(\ZZ)$. The normalization is necessary to ensure that $\delta_\ZZ^0$ is the unit of $\Ch(\ZZ)$, cf. (\ref{sss}ii). It is readily verified that, for each $n\geq 0$ and $k\geq 0$, the chain map $\xi_k(\delta_\ZZ)^n\to\xi_k(\delta_\ZZ)^0$ is surjective and hence a fibration for the projective model structure on $\Ch(\ZZ)$. A closer look at the cosimplicial chain complex $\xi_k(\delta_\ZZ)^\cdot$ reveals that, for each $n\geq 1$, the matching map $\xi_k(\delta_\ZZ)^n\to M_n(\xi_k(\delta_\ZZ)^\cdot)$ is surjective, i.e. a fibration in $\Ch(\ZZ)$. This implies that the cosimplicial chain map $\xi_k(\delta_\ZZ)^\cdot\to\xi_k(\delta_\ZZ)^0$ is a Reedy fibration in $\Ch(\ZZ)^\Delta$. Moreover, it follows from (\ref{sss}iii) and (\ref{mainexamples}b) that $\xi_k(\delta_\ZZ)^n\to\xi_k(\delta_\ZZ)^0$ is a weak equivalence, i.e. $\xi_k(\delta_\ZZ)^\cdot\to\xi_k(\delta_\ZZ)^0$ is a trivial Reedy fibration in $\Ch(\ZZ)^\Delta$. Therefore, $\LL$ is \emph{strongly $\delta_\ZZ$-reductive} and Theorem \ref{main} implies that the $E_\infty$-chain operad $\Coend_\LL(\delta_\ZZ)$ is filtered by \emph{$E_m$-chain operads} $\Coend_{\LL_m}(\delta_\ZZ)$. This filtration induces the standard filtration of the \emph{surjection operad} $\XX$ of \cite{BF,MS2} by means of the following embedding of $\XX$ into $\Coend_\LL(\delta_\ZZ)$.

Observe first that the chain complex $\xi_k(\delta_\ZZ)^0$ is isomorphic to $\XX(k)$ since $\delta_\ZZ$-realization of the multisimplicial set $\LL(-,\dots,-;0)$ yields exactly one generator $N_*(\Delta[n_1];\ZZ)\otimes\cdots\otimes N_*(\Delta[n_k];\ZZ)$ of dimension $n_1+\cdots+n_k$ for each integer-string containing $n_1+1$ times $1$, $\dots$, $n_k+1$ times $k$, \emph{without repetitions}. The induced ``prismatic'' face operators agree (up to a sign) with the one defined in \cite{BF,MS2}. The operad structure on $\XX$ as well as the embedding of $\XX$ into $\Coend_\LL(\delta_\ZZ)$ follow now from the existence of cosimplicial chain maps $\xi_k(\delta_\ZZ)^0\otimes N_*(\Delta[n];\ZZ)\to\xi_k(\delta_\ZZ)^n$ compatible with the right $\Coend_\LL(\delta_\ZZ)$-module structure of $\xi_k(\delta_\ZZ)^\cdot$. The image of a generator $x\in\xi_k(\delta_\ZZ)^0$ under this chain map is obtained by decomposing the integer-string representing $x$ in all possible ways into $n+1$ \emph{overlapping} substrings, putting $n$ separating vertical bars between these substrings, and taking the sum over all these subdivided integer-strings. For instance, for $n=1$, the generator $121\in\xi_2(\delta_\ZZ)^0$ is taken to $1|121+12|21+121|1\in\xi_2(\delta_\ZZ)^1$. It is tedious but straightforward to verify the cosimplicial identities and the compatibility with the operad structure of $\LL$. The resulting operad structure on $\XX=\xi(\delta_\ZZ)^0$ coincides (up to a sign) with the one constructed in \cite{BF,MS2}. The $m$-th filtration stage $\XX_m$ of $\XX$ is induced by the $m$-th filtration stage $\Coend_{\LL_m}(\delta_\ZZ)$ of $\Coend_\LL(\delta_\ZZ)$ since both are defined by complexity. In particular, we get a weak equivalence of $E_m$-chain operads $\XX_m\to\Coend_{\LL_m}(\delta_\ZZ)$ admitting a canonical retraction in the category of collections, cf. also \cite{BBM}.\end{exms}

\subsection{Conormalization}\label{conorm}The Dold-Kan correspondence establishes an equivalence of categories between simplicial abelian groups and chain complexes. In this correspondence, the chain complex associated to a simplicial abelian group $X_\cdot$ is \emph{normalized} and hence isomorphic to the coend $X_\cdot\otimes_\Delta\delta^\cdot_\ZZ$. Dually, a cosimplicial abelian group $X^\cdot$ has an associated normalized cochain complex. This normalized cochain complex is isomorphic to the end $\uHom_\Delta(\delta_\ZZ^\cdot, X^\cdot)$ provided the bifunctor $\uHom_\Delta$ is understood to take values in differential $\ZZ$-graded abelian groups. In other words: conormalization can be viewed as $\delta_\ZZ$-totalization.\vspace{1ex}

The following statement is Deligne's conjecture for Hochschild cochains; for alternative proofs see \cite{BF,Ka1,KS1,MS1,MS2,Ta,V}.

\begin{thm}\label{Deligne}The normalized Hochschild cochain complex of a unital associative algebra carries a canonical action by an $E_2$-chain operad.\end{thm}

\begin{proof}The endomorphism operad $\End_A$ of an algebra $A$ is a multiplicative non-symmetric operad in abelian groups, and is therefore by Proposition \ref{L2} an $\LL_2$-algebra in abelian groups. The unary part of $\LL_2$ endow $\End_A$ with the structure of a cosimplicial abelian group. It follows from \cite{GV} that conormalization of this cosimplicial abelian group yields the \emph{normalized Hochschild cochain complex} $\HCC^*(A;A)$. On the other hand, conormalization of a cosimplicial abelian group can be obtained as $\delta_\ZZ$-totalization (\ref{conorm}). Therefore, we have $\HCC^*(A;A)=\uHom_\Delta(\delta_\ZZ,\End_A).$ Theorem \ref{main} induces then through condensation (\ref{condensation}) a canonical action on $\HCC^*(A;A)$ by the $E_2$-chain operad $\Coend_{\LL_2}(\delta_\ZZ)$. (This action restricts to the second filtration stage $\XX_2$ of the surjection operad $\XX$, cf. (\ref{mainexamples}c) and \cite{BF,MS2}.)\end{proof}

\section{\label{cyclic}The cyclic lattice path operad}

The dual description of $\LL$ obtained in (\ref{duallattice}) suggests the definition of a cyclic version $\LL^{cyc}$ of the lattice path operad. Recall from Krasauskas \cite{Kr} and Fiedorowicz-Loday \cite{FL} that a \emph{crossed simplicial group} is a category  $\Delta G$ such that\begin{itemize}\item[(i)]$\Delta G$ contains $\Delta$ as a wide\footnote{i.e. $\Delta$ and $\Delta G$ have same objects.} subcategory;\item[(ii)]each morphism $\psi\in\Delta G([m],[n])$ factors uniquely as an automorphism $g\in\Aut_{\Delta G}([m],[m])$ followed by a morphism $\phi\in\Delta([m],[n])$.\end{itemize}In particular, the family of groups $G_{[n]}=\Aut_{\Delta G}([n],[n])$ inherits the structure of a \emph{simplicial set} by means of the following commutative squares in $\Delta G$:\begin{diagram}[small][m]&\rTo^\phi&[n]\\\dTo^{\phi^*(g)}&&\dTo_g\\[m]&\rTo_{g_*(\phi)}&[n],\end{diagram}in which the horizontal morphisms belong to $\Delta$ and the vertical morphisms are automorphisms. It follows from the definition (\ref{noncomm}) that the category $\Delta\Sg$ has the structure of a crossed simplicial group with $\Sg_{[n]}=\Aut_{\Delta\Sg}([n],[n])=\Sg_{n+1}$. In between $\Delta$ and $\Delta\Sg$ sits another crossed simplicial group $\Delta C$ whose automorphism group $C_{[n]}=\Aut_{\Delta C}([n],[n])=\ZZ/(n+1)\ZZ$ is the subgroup of $\Sg_{n+1}$ generated by the cycle $0\mapsto 1\mapsto\cdots\mapsto n\mapsto 0$. Indeed, for any $g\in C_{[n]}$ and $\phi\in\Delta([m],[n])$, the induced automorphism $\phi^*(g)$ in $\Delta\Sg$ belongs to $C_{[m]}$. As a simplicial set, $(C_{[n]})_{n\geq 0}$ is isomorphic to the simplicial circle $S^1=\Delta[1]/\partial\Delta[1]$. As a category, $\Delta C$ is isomorphic to Connes' \emph{cyclic category} $\Lambda$, cf. \cite{Co,FT,FL,Kr}.

\begin{dfn}The \emph{cyclic lattice path operad} $\LL^{cyc}$ is the $\NN$-coloured operad in sets defined by$$\LL^{cyc}(n_1,\dots,n_k;n)=\{x\in\Delta\Sigma_+([n_1]*\cdots*[n_k],[n])\,|\,x_i\in\Delta C([n_i],[n]),i=1,\dots,k\},$$where the operad substitution maps are given by join and composition in $\Delta\Sigma_+$.\end{dfn}

The underlying category $\LL_u^{cyc}$ is the cyclic category $\Delta C$. Moreover, the inclusion $\Delta\inc\Delta C$ induces an inclusion of $\NN$-coloured operads $\LL\inc\LL^{cyc}$. As in (\ref{pullback}) the cyclic lattice path operad may be identified with the pullback\begin{diagram}[small]\LL^{cyc}&\rTo&\EE_{\Delta\Sg_+}\\\dTo&&\dTo\\\FF_{\Delta C}&\rTo&\FF_{\Delta\Sigma}\end{diagram} in the category of coloured operads. The elements of $\LL^{cyc}(n_1,\dots,n_k;n)$ can be described as follows. We have a bijection$$\LL(n_1,\dots,n_k;n)\times C_{[n_1]}\times\cdots\times C_{[n_k]}\cong\LL^{cyc}(n_1,\dots,n_k;n)$$obtained by join and composition. Therefore, any element of $\LL^{cyc}(n_1,\dots,n_k;n)$ determines, and is determined by, an $n$-fold subdivided integer-string containing (for each $i$) $n_i+1$ times the integer $i$, together with (for each $i$) a cyclic automorphism in $C_{[n_i]}$. The latter will be represented by distinguishing one among the $n_i+1$ occurences of $i$ (namely the one which becomes the leading element after application of the cyclic automorphism). For instance, $1|\underline{2}\underline{1}|\underline{3}|123\in\LL^{cyc}(2,1,1;3)$ and $2|\underline{1}|\underline{2}12\in\LL^{cyc}(1,2;2)$ are typical ``cyclic'' lattice paths, where the distinguished integers have been underlined. The operad substitution maps of $\LL^{cyc}$ are defined by renumbering and substitution, like for $\LL$, except that the leftmost substring is now substituted into the distinguished integer, and the subsequent substrings are substituted in a circular order. For instance:$$1|\underline{2}\underline{1}|\underline{3}|123\circ_12|\underline{1}|\underline{2}12=\underline{2}12|\underline{3}2|\underline{4}|\underline{1}34\in\LL^{cyc}(1,2,1,1;3).$$This yields as a byproduct a combinatorial description of the cyclic category $\Delta C$ extending the one of the simplex category $\Delta$, mentioned in (\ref{duality}).

\subsection{The cyclic complexity index}\label{cyccomp}The definition of the complexity index for cyclic lattice paths is slightly trickier than for the ordinary lattice paths, since this index has to be preserved by the $C_{[n]}$-action on $\LL^{cyc}(n_1,\dots,n_k;n)$. Fortunately, the individual complexity indices $c_{ij}(x)$ of a lattice path $x\in\LL(n_1,\dots,n_k;n)$ have the property that if $c_{ij}(x)$ is \emph{even}, then any cyclic permutation of the integer-string either preserves the complexity index or \emph{lowers} it by one. Therefore, it makes sense to define the complexity index of a cyclic lattice path $x\in\LL^{cyc}(n_1,\dots,n_k;n)$ to be the lowest \emph{even} integer $c_{cyc}(x)$ that serves as an upper bound for all individual complexity indices $c_{ij}(g.x)$ of all $g.x,\,g\in C_{[n]}$. In this way we obtain an even filtration of $\LL^{cyc}$ by $\NN$-coloured suboperads $\LL^{cyc}_0\subset\LL^{cyc}_2\subset\LL^{cyc}_4\subset\cdots$. The underlying category of each filtration stage $\LL_{2m}^{cyc}$ is the cyclic category $\Delta C$, and $\LL^{cyc}_0=\Delta C$.

\subsection{Cyclic operads}\label{cyclicoperad}

For our purposes we need a planar, non-symmetric version of Getzler-Kapranov's \emph{cyclic operads}, cf. \cite{GK}. Recall that a non-symmetric operad $P=(P(n))_{n\geq 0}$ in $\EE$ induces for each \emph{planar rooted tree} $T$ an operad composition map $m_T:P(|v_0|)\otimes P(|v_1|)\otimes\cdots\otimes P(|v_k|)\to P(|T|)$, where $v_0$ is the root-vertex, $v_1,\dots,v_k$ are the other internal vertices of $T$, and $|v|$ (resp. $|T|$) denotes the number of edges incident to $v$ minus $1$ (resp. the number of external edges of $T$ minus $1$). We assume $T$ is equipped with an external root-edge (incident to $v_0$). The operad composition maps $m_T$ have to be associative with respect to grafting of planar rooted trees, and unital with respect to the given unit $I_\EE\to P(1)$. Such a planar rooted tree can be considered as a composition scheme where the external non-root-edges represent the inputs, the root-edge represents the output, and each vertex $v$ represents an individual composition of $|v|$ inputs yielding one output. A cyclic structure on a non-symmetric operad $P$ amounts to the possibility of changing cyclically the role of inputs/output for each of the $m_T$.

More precisely, a non-symmetric operad $P$ will be called \emph{cyclic}, if each $P(n)$ comes equipped with an action by the cyclic group $C_{[n]}=\ZZ/(n+1)\ZZ$ such that the operad composition maps of $P$ are compatible with the cyclic group actions in the following sense: The planar structure of a planar tree $T$ induces a cyclic order on the $|T|+1$ external edges. Each of the external edges of a planar tree can serve as root-edge yielding in general $|T|+1$ different planar rooted trees. Therefore, the cyclic group $C_{[n]}$ acts canonically on the set of planar rooted trees with $n+1$ external edges, by circular ``change of root-edge''. We adopt the convention that the cyclic group $C_{[n]}$ shifts the root-edge in clockwise circular order (i.e. the planar tree itself is rotated counterclockwise). For $g\in C_{[n]}$ and a planar rooted tree $T$ with $n+1$ external edges, we denote by $g.T$ the planar rooted tree obtained by such a $g$-shift of the root-edge. Observe that in $g.T$ each vertex gets an individual root-edge, namely the $v$-incident edge closest to the root-edge of $g.T$. Therefore, the cyclic permutation $g\in C_{[n]}$ induces for each vertex $v$ of $T$ a cyclic permutation $g_v\in C_{[|v|]}$ which takes the root-edge of $v$ in $T$ to the root-edge of $v$ in $g.T$.

The equivariance condition for a cyclic operad $P$ reads now as follows. For each planar rooted tree $T$ with vertices $v_0,v_1,\dots,v_k$ and each $g\in C_{[|T|]}$ one has:$$g\circ m_T=m_{g.T}\circ(g_{v_0}\otimes g_{v_1}\otimes\cdots\otimes g_{v_k}).$$The composition maps $m_T$ of a non-symmetric operad $P$ are generated by those $m_T$ whose tree $T$ is obtained by grafting an $n$-bunch onto the $i$-th input-edge of an $m$-bunch. This composition map is usually denoted $\circ_i:P(m)\otimes P(n)\to P(m+n-1)$. Let $\tau_{m}$ denote the standard generator of $C_{[m]}$. The equivariance conditions of a cyclic operad require then for any $x\in P(m),\,y\in P(n)$ that $$\tau_{m+n-1}(x\circ_i y)=\begin{cases}\tau_n(y)\circ_n\tau_m(x)&\text{for}\quad i=1,\\\tau_m(x)\circ_{i-1}y&\text{for}\quad i>1,\end{cases}$$and that the unit $I_\EE\to P(1)$ be fixed under $C_{[1]}$. Conversely, the relations involving the $\circ_i$-products imply all other relations of a cyclic operad. The non-symmetric operad $\Ass$ for associative monoids carries a (trivial) cyclic structure. Following Menichi \cite{Me} and Salvatore \cite{Sa}, a cyclic operad $P$ will be called \emph{multiplicative} if $P$ comes equipped with a morphism of cyclic operads $\Ass\to P$.

A non-trivial example of a cyclic operad in topological spaces is given by the family $(\Aa_n)_{n\geq 0}$ of Stasheff polytopes (often called \emph{associahedra}) whose algebras are the $A_\infty$-algebras. We shall see in Lemma \ref{endo} below that the Hochschild cochain complex of a symmetric Frobenius algebra is a multiplicative cyclic operad in chain complexes. We are now ready to state the following cyclic analog of (\ref{L2}):

\begin{prp}\label{L2cyc}Let $\EE$ be a cocomplete, closed symmetric monoidal category. The category of $\LL^{cyc}_2$-algebras in $\,\EE$ is isomorphic to the category of multiplicative cyclic operads in $\EE$.\end{prp}

\begin{proof}We use the same notations as in the proof of (\ref{L2}). The $\NN$-coloured operad $\OO^{cyc}$ for multiplicative cyclic operads derives from the $\NN$-coloured operad $\OO$ for multiplicative non-symmetric operads by the formula: $$\OO^{cyc}(n_1,\dots,n_k;n)=\OO(n_1,\dots,n_k;n)\times C_{[n_1]}\times\cdots\times C_{[n_k]}.$$The $k$-tuple of cyclic permutations represents a choice of ``root-edge'' for each of the $k$ labelled vertices $v(i)$ of the planar rooted tree $T\in\OO(n_1,\dots,n_k;n)$ (such an element acts by first letting act the cyclic permutations followed by composition according to the planar tree structure). The substitution maps for $\OO^{cyc}$ are defined in the same way as for $\OO$ by a substitution of trees into labelled vertices of a tree; this time however, the root-edge of the substituted tree must match with the distinguished root-edge of the vertex, and the other external edges are grafted in a circular order, using the planar structure of both trees. In particular we get\begin{align*}\OO^{cyc}(n_1,\dots,n_k;n)&=\OO(n_1,\dots,n_k;n)\times C_{[n_1]}\times\cdots\times C_{[n_k]}\\&=\LL_2(n_1,\dots,n_k;n)\times C_{[n_1]}\times\cdots\times C_{[n_k]}\\&=\LL_2^{cyc}(n_1,\dots,n_k;n).\end{align*}It remains to be shown that the $\NN$-coloured operad structures coincide as well. The only hint we give here is the observation that in the bijection (\ref{L2}) between subdivided integer-strings and labelled planar rooted trees, the occurences of $i$ in the integer-string $x$ correspond geometrically to the sectors around the vertex $v_x(i)$ in $T_x$ delimited by the edges incident to $v_x(i)$; the circular order of these sectors corresponds to the circular order of the occurences of $i$ in the integer-string $x$.\end{proof}

\subsection{Framed $E_2$-operads}\label{semi-direct}The standard topological construction of a framed $E_2$-operad is an enhanced version of the little disks operad $\DD=(\DD(k))_{k\geq 0}$, the so-called \emph{framed little disks operad} $f\DD=(f\DD(k))_{k\geq 0}$, introduced by Getzler in \cite{Get}. Recall that points of $\DD(k)$ (resp. $f\DD(k)$) are $k$-tuples of ``axial'' (resp. ``conformal'') disk-embeddings $D\to D$, whose images have pairwise disjoint interiors; here ``axial'' means any composite of translation and dilation, while ``conformal'' means any composite of translation, dilation and rotation. The operad substitution maps are induced by composition of embeddings. The projections identify $\DD(k)$ (resp. $f\DD(k)$) with a subspace of $\DD(1)^k$ (resp. $f\DD(1)^k$). A \emph{framed $E_2$-operad} is any \emph{$\Sg$-cofibrant} \cite{BM} topological operad weakly equivalent to the framed little disks operad. By a framed $E_2$-operad in chain complexes, we mean any \emph{$\Sg$-cofibrant} chain operad weakly equivalent to the singular chain complex of the framed little disks operad.

In order to \emph{recognize} framed $E_2$-operads, it is useful to observe that $f\DD$ derives from $\DD$ by a general categorical construction, which has been described by Markl \cite{Ml} and Salvatore-Wahl \cite{SW}. Indeed, $f\DD$ is a \emph{semi-direct product} $\DD\rtimes SO(2)$, where $\DD$ is considered as an operad in $SO(2)$-spaces. The group $SO(2)$ acts on $\DD(1)$ by conjugation; the resulting diagonal $SO(2)$-action on $\DD(1)^k$ restricts to $\DD(k)$, and it is readily verified that the substitution maps of the little disks operad are $SO(2)$-equivariant. One has then $f\DD(k)=\DD(k)\times SO(2)^k$, and the substitution maps of $f\DD$ are completely determined by the substitution maps of $\DD$ and the $SO(2)$-action on $\DD$, cf. \cite{Ml,SW}. One way of understanding the semi-direct product construction is to observe that the map $SO(2)\times\DD(k)\to\DD(k)\times SO(2)^k$ given on the first factor by the $SO(2)$-action, and on the second factor by the diagonal $SO(2)\to SO(2)^k$, induces a \emph{distributive law} (in the sense of Beck \cite{Beck}) between the monads associated to $\DD$ and to an $SO(2)$-action. The resulting composite monad is then precisely the monad associated to $f\DD$; in particular, the category of $f\DD$-algebras is isomorphic to the category of $\DD$-algebras in $SO(2)$-spaces, cf. \cite{Beck,Ml,SW}.

\subsection{The standard cocyclic objects}\label{cocyclic}We shall need a certain amount of the theory of (co)cyclic objects. The inclusion $i:\Delta\to\Delta C$ induces an adjunction$$i_!:\Set^{\Delta^\op}\lra\Set^{\Delta C^\op}:i^*.$$This adjunction identifies cyclic sets with $i^*i_!$-algebras in simplicial sets; thus, a cyclic set is a simplicial set $X$ with some algebraic extra-structure. It turns out (see \cite[Theorem 5.3]{FL} where $i^*i_!$ is denoted $F_{C_*}$) that this extra-structure induces a canonical $SO(2)$-action on the realization $|X|_{\delta_{top}}$ of the simplicial set $X$. In general, the realization of the simplicial set underlying a crossed simplicial group $\Delta G$ is a \emph{topological group}. In our particular case, the group structure on $|S^1|_{\delta_{top}}$ can be visualized as follows: The real line $\RR$ is the realization of the simplicial set $(\ZZ\times\Delta[1])/((n,1)\sim(n+1,0))_{n\in\ZZ}$. This simplicial set underlies a crossed simplicial group, which we shall denote by $\Delta\ZZ$, since $\Aut_{\Delta\ZZ}([n],[n])=\ZZ$ for all $n$. The crossed simplicial group structure encodes the $\ZZ$-action on $\RR$ by integer translation. Therefore, the realization of $\Delta\ZZ$ as a crossed simplicial group is the topological group $(\RR,+)$. It follows then from the short exact sequence of crossed simplicial groups $1\to\ZZ\to\Delta\ZZ\to\Delta C\to 1$ that the realization of $\Delta C$ as a crossed simplicial group is isomorphic to $(\RR/\ZZ,+)$, i.e. to $\,SO(2)$, cf. \cite[Remark 3.9]{FL}.\vspace{1ex}

We shall use the following two cocyclic objects:\begin{align*}\delta^{cyc}_{top}:\Delta C\to\Set^{\Delta C^\op}\overset{i^*}{\to}\Set^{\Delta^\op}&\overset{|-|_{\delta_{top}}}{\to}\Top\\\delta^{cyc}_\ZZ:\Delta C\to\Set^{\Delta C^\op}\overset{i^*}{\to}\Set^{\Delta^\op}&\overset{N_*(-;\ZZ)}{\to}\Ch(\ZZ)\end{align*}

We have $(\delta_{top}^{cyc})^m=\Delta_m\times SO(2)$, cf. \cite[Theorem 3.4]{J}. A cyclic operator $\phi\in\Delta C([m],[n])$ acts on the first factor by affine extension of the natural vertex-action, and on the second factor as rotation by $2\pi k/m$, where $k$ is determined by the factorization of $\phi$ into a cyclic permutation followed by a simplicial operator.

\begin{lem}\label{leftKan}The standard cocyclic space $\delta_{top}^{cyc}$ (resp. chain complex $\delta_\ZZ^{cyc}$) may be identified with the left Kan extension of the standard cosimplicial space $\delta_{top}$ (resp. chain complex $\delta_\ZZ$) along the canonical inclusion $i:\Delta\to\Delta C$.\end{lem}

\begin{proof}Since as well $|\!-\!|_{\delta_{top}}$ as well $N_*(-;\ZZ)$ are left adjoints and therefore preserve left Kan extensions, it is sufficent to show that the left Kan extension of the Yoneda-embedding $\delta_{yon}:\Delta\to\Set^{\Delta^\op}$ may be identified with the corestricted cyclic Yoneda-embedding $i^*\circ\delta_{yon}^{cyc}:\Delta C\to\Set^{(\Delta C)^\op}\to\Set^{\Delta^\op}$. This follows from the pointwise formula of the left Kan extension along $i:\Delta\to\Delta C$.\end{proof}

\begin{lem}\label{Connes}(a) The (normalized) chain complex $N_*(S^1;\ZZ)$ of the simplicial circle is an exterior algebra on one generator of degree $1$. The category of $N_*(S^1;\ZZ)$-modules is isomorphic to the category of \emph{mixed complexes} (i.e. $\NN$-graded abelian groups with differentials $b$ and $B$ of degree $-1$ and $+1$ such that $Bb+bB=0$).\vspace{1ex}

(b) The free $N_*(S^1;\ZZ)$-module on the chain complex of a simplicial set $X$ is chain-homotopy equivalent to the chain complex $N_*(i^*i_!X;\ZZ)$ of the free cyclic set generated by $X$. The chain equivalence is induced by an Eilenberg-Zilber map.\end{lem}

\begin{proof} (a) is classical. For (b) observe that the free cyclic set $i^*i_!X$ generated by $X$ is a ``crossed'' product $X\bowtie S^1$ for which the classical Eilenberg-Zilber shuffle map $N_*(X;\ZZ)\otimes N_*(S^1;\ZZ)\to N_*(X\bowtie S^1;\ZZ)$ remains well defined.\end{proof}

\subsection{Recognition principle for (framed) $E_2$-operads}\label{recognition}Since the (framed) little disks operad is aspherical in each arity, there is an intrinsic way to characterize (framed) $E_2$-operads, due to Fiedorowicz \cite{F} and Salvatore-Wahl \cite{SW}. 

Denote by $\BB_k$ (resp. $\RB_k$) the braid (resp. ribbon-braid) group on $k$ strands. In particular, one has $\BB_k\cong\pi_1(\DD(k)/\Sg_k)$ and $\RB_k\cong\pi_1(f\DD(k)/\Sg_k)$. A \emph{braided} (resp. \emph{ribbon-braided}) operad $\tilde{\OO}$ in a monoidal model category $\EE$ with cofibrant unit $I_\EE$ is a sequence of objects $\tilde{\OO}(k)$ of $\EE$, equipped with a unit $I_\EE\to\tilde{\OO}(1)$, with actions of $\BB_k$ (resp. $\RB_k$) on $\tilde{\OO}(k)$, and with operad substitution maps which are equivariant with respect to canonical substitution maps $\BB_k\times\BB_{i_1}\times\cdots\times\BB_{i_k}\to\BB_{i_1+\cdots+i_k}$ (resp. $\RB_k\times\RB_{i_1}\times\cdots\times\RB_{i_k}\to\RB_{i_1+\cdots+i_k}$). We require that $\tilde{\OO}(0)=I_\EE$ and that $\tilde{\OO}(k)$ is cofibrant for the projective model structure on $\BB_k$- (resp. $\RB_k$)-objects in $\EE$.

Given an epimorphism of discrete groups $\phi:G\to H$, a morphism $f:X\to Y$ in $\EE$ is called a \emph{$\phi$-covering}, whenever $X$ is a $G$-object, $Y$ is an $H$-object and $f:X\to Y$ is a \emph{$\phi$-equivariant fibration} in $\EE$ such that $X/Ker(\phi)\cong Y$ via $f$. A \emph{(ribbon-)braided covering} of a symmetric operad $\OO$ consists of a (ribbon-)braided operad $\tilde{\OO}$ together with a map of (non-symmetric) operads $\tilde{\OO}\to\OO$ inducing for each $k$ a $(\BB_k\to\Sg_k)$- resp. ($\RB_k\to\Sg_k$)-covering $\tilde{\OO}(k)\to\OO(k)$. A (ribbon-)braided covering is called \emph{universal} if, for each $k$, the object $\tilde{\OO}(k)$ is weakly contractible, i.e. the canonical augmentation $\tilde{\OO}(k)\to\tilde{\OO}(0)$ is a weak equivalence in $\EE$.

The recognition principles of Fiedorowicz and Salvatore-Wahl read as follows: \emph{a topological operad $\OO$ is $E_2$ (resp. framed $E_2$) if and only if $\OO$ admits a universal braided (resp. ribbon-braided) covering}, cf. \cite[Example 3.1]{F}, \cite[Theorem 7.3]{SW}. For a general monoidal model category $\EE$ satisfying the axioms of (\ref{sss}), the existence of a universal (ribbon-)braided covering is a stronger condition than being a (framed) $E_2$-operad in the sense of (\ref{sss}) and (\ref{semi-direct}). For instance, the singular chain complex of the (framed) little disks operad admits a universal (ribbon-)braided covering in the monoidal model category of chain complexes, but there might be weakly equivalent $\Sg$-cofibrant chain operads which do not admit such a covering.

It is however important to observe that under the hypotheses of Theorem \ref{main}, the $E_2$-operad $\Coend_{\LL_2}(\delta)$ admits a universal braided covering in $\EE$. Indeed, as shown in the proof of Theorem \ref{main}, there is a zigzag of weak equivalences\begin{diagram}[small]\Coend_{\LL_2}(\delta)&\lTo^\sim&\Coend_{\widehat{\LL}_2}(\delta)&\rTo^\sim B_\delta\cgr_2\end{diagram}relating  $\delta$-condensation of $\LL_2$ and $\delta$-realization of the categorical $E_2$-operad $\cgr_2$. Since there is a general product-preserving construction of the universal covering category of a category, there exists a universal braided covering $\tilde{\cgr}_2$ of the operad $\cgr_2$. The latter induces universal braided coverings in $\EE$ for the whole zizag above, replacing the relevant colimits and homotopy colimits over $\cgr_2(k)$ by the corresponding colimits and homotopy colimits over the universal covering $\tilde{\cgr}_2(k)$.

There is a convenient method of constructing universal ribbon-braided coverings, based on the identification $\RB=\BB\wr\ZZ$, i.e. $\RB_k=\BB_k\ltimes\ZZ^k$, cf. Getzler \cite{Get}. Here $\BB_k$ acts on $\ZZ^k$ by permuting the factors via $\BB_k\to\Sg_k$. We put $\RS=\Sg\wr\ZZ$, i.e. $\RS_k=\Sg_k\ltimes\ZZ^k$, and define a \emph{ribbon-symmetric} operad $\tilde{\OO}$ like above, but with $\RS_k$-actions on $\tilde{\OO}(k)$. A \emph{$\ZZ$-covering} $\tilde{\OO}\to\OO$ is then a map of non-symmetric operads such that $\tilde{\OO}$ is ribbon-symmetric and $\tilde{\OO}(k)\to\OO(k)$ is an $(\RS_k\to\Sg_k)$-covering for each $k$. Then a universal ribbon-braided covering $\tilde{\tilde{\OO}}\to\OO$ can be obtained as the composite of a universal braided covering $\tilde{\tilde{\OO}}\to\tilde{\OO}$ and a $\ZZ$-covering $\tilde{\OO}\to\OO$.

\begin{thm}\label{cyclcond}Condensation of $\LL_2^{cyc}$ with respect to $\delta_{top}^{cyc}$ (resp. $\delta_\ZZ^{cyc}$) yields a framed $E_2$-operad in topological spaces (resp. chain complexes).\end{thm}

\begin{proof}In the topological case we get, by (\ref{leftKan}) and by adjunction,\begin{align*}\Coend_{\LL_2^{cyc}}(\delta_{top}^{cyc})(k)&=\uHom_{\Delta C}(\delta_{top}^{cyc},|\LL_2^{cyc}(-,\cdots,-;\cdot)|_{\delta_{top}^{cyc}})\\&=\uHom_{\Delta C}(i_!\delta_{top},|\LL_2^{cyc}(-,\cdots,-;\cdot)|_{\delta_{top}^{cyc}})\\&=\uHom_\Delta(\delta_{top},i^*|\LL_2^{cyc}(-,\cdots,-;\cdot)|_{\delta_{top}^{cyc}}).\end{align*}
The multicyclic set $\LL_2^{cyc}(-,\cdots,-;n)=\LL_2(-,\cdots,-;n)\times C_{[-]}\times\cdots\times C_{[-]}$ is the \emph{free} multicyclic set generated by the multisimplicial set $\LL_2(-,\cdots,-;n)$. A multivariate version of \cite[Theorem 5.3]{FL} implies that its $\delta_{top}$-realization is isomorphic to $|\LL_2(-,\cdots,-;n)|_{\delta_{top}}\times SO(2)^k$. Therefore we get$$\Coend_{\LL_2^{cyc}}(\delta_{top}^{cyc})(k)=\Coend_{\LL_2}(\delta_{top})(k)\times SO(2)^k,\quad k\geq 0.$$
We have seen in (\ref{cocyclic}) that the short exact sequence of crossed simplicial groups $\ZZ\to\Delta\ZZ\to\Delta C$ induces upon realization the $\ZZ$-principal fibration $\ZZ\to\RR\to\RR/\ZZ$. Therefore, identifying $\RR/\ZZ$ with $SO(2)$ yields a sequence of maps$$\Coend_{\LL_2}(\delta_{top})(k)\times \RR^k\rightarrow\Coend_{\LL^{cyc}_2}(\delta_{top}^{cyc})(k),\quad k\geq 0.$$There is a unique structure of \emph{ribbon-symmetric} operad on the left hand side such that the maps define a \emph{$\ZZ$-covering} of operads, cf.  (\ref{recognition}). Since $\Coend_{\LL_2}(\delta_{top})$ is a topological $E_2$-operad, cf. (\ref{mainexamples}a), whose $E_1$-suboperad $\Coend_{\LL_1}(\delta_{top})$ is also a suboperad of the left hand side above, Fiedorowicz's argument  \cite{F} gives a universal braided covering of the left hand side. By composition, this induces a universal ribbon-braided covering of the right hand side, which is thus a framed $E_2$-operad.

In the chain complex case, we get like above$$\Coend_{\LL_2^{cyc}}(\delta_\ZZ^{cyc})(k)=\uHom_\Delta(\delta_\ZZ,i^*|\LL_2^{cyc}(-,\cdots,-;\cdot)|_{\delta_\ZZ^{cyc}}).$$Again using that $\LL_2^{cyc}(-,\cdots,-;n)$ is the free multicyclic set generated by the multisimplicial set $\LL_2(-,\cdots,-;n)$, we can write its $\delta_\ZZ$-realization as a coend$$|\LL_2^{cyc}(-,\cdots,-;n)|_{\delta_\ZZ^{cyc}}=\LL_2(-,\cdots,-;n)\otimes_{\Delta\times\cdots\times\Delta}j^*j_!(\delta_{\ZZ}\otimes\cdots\otimes\delta_\ZZ)$$where $j=i\times\cdots\times i:\Delta\times\cdots\times\Delta\to\Delta C\times\cdots\times\Delta C$. By (\ref{leftKan}) and (\ref{Connes}b), the chain complex $(i^*i_!\delta_\ZZ)^n=(i^*\delta_\ZZ^{cyc})^n=N_*(i_*i_!\Delta[n];\ZZ)$ is chain-homotopy equivalent to $N_*(\Delta[n];\ZZ)\otimes N_*(S^1;\ZZ)$, hence the cosimplicial chain complex $i^*i_!\delta_\ZZ$ is levelwise chain-homotopy equivalent to $\delta_\ZZ\otimes N_*(S^1;\ZZ)$. A multivariate version of this yields a levelwise chain-homotopy equivalence between $j^*j_!(\delta_\ZZ\otimes\cdots\otimes\delta_\ZZ)$ and $\delta_\ZZ\otimes\cdots\otimes\delta_\ZZ\otimes N_*(S^1;\ZZ)^{\otimes k}$. Therefore, the coend above is chain-homotopy equivalent to $|\LL_2(-,\cdots,-;n)|_{\delta_\ZZ}\otimes N_*(S^1;\ZZ)^{\otimes k}$. Since $\delta_\ZZ$-totalization is conormalization, see (\ref{conorm}), we finally get chain-homotopy equivalences$$\Coend_{\LL_2^{cyc}}(\delta_\ZZ^{cyc})(k)\simeq\Coend_{\LL_2}(\delta_\ZZ)(k)\otimes N_*(S^1)^{\otimes k},\quad k\geq 0.$$The components of the right hand side form a chain suboperad of $\Coend_{\LL_2^{cyc}}(\delta_\ZZ^{cyc})$. Since $\Coend_{\LL_2}(\delta_\ZZ)$ is an $E_2$-chain operad, cf. (\ref{mainexamples}c), the end of proof is similar to the topological case, and uses instead that the sequence of maps$$\Coend_{\LL_2}(\delta_\ZZ)(k)\otimes N_*(\RR)^{\otimes k}\rightarrow\Coend_{\LL_2}(\delta_\ZZ)(k)\otimes N_*(S^1)^{\otimes k},\quad k\geq 0,$$defines a $\ZZ$-covering of chain operads whose domain contains the $E_1$-chain operad $\Coend_{\LL_1}(\delta_\ZZ)$ as a chain suboperad.\end{proof}

\begin{rmk}Using (\ref{condensation}), Theorem \ref{cyclcond} and Proposition \ref{L2cyc} establish the \emph{topological cyclic Deligne conjecture}: the totalization of a multiplicative cyclic operad in spaces admits an action by a framed $E_2$-operad. This is the main result of Salvatore's recent work \cite{Sa}. It follows from (\ref{mainexamples}a) that his construction of a framed $E_2$-action is isomorphic to ours. His proof that the constructed operad is indeed a framed $E_2$-operad is similar to the one presented above. Voronov's construction of a \emph{cacti operad} action is also closely related, cf. \cite{Ka2,Sa,V2}. 

In the chain complex case, there is a chain suboperad $f\XX_2$ of $\Coend_{\LL_2^{cyc}}(\delta_\ZZ^{cyc})$, a ``framed analog'' of the second filtration stage $\XX_2$ of the surjection suboperad $\XX$ of $\Coend_{\LL}(\delta_\ZZ)$, cf. (\ref{mainexamples}c). This framed $E_2$-chain operad $f\XX_2$ is closely related to Sullivan's \emph{chord diagrams}, cf. Tradler-Zeinalian \cite{TZ}.\end{rmk}


\subsection{Symmetric Frobenius monoids}\label{symmFrob}A \emph{Frobenius monoid} $(A,\mu,\eta,\phi)$ in a closed symmetric monoidal category $(\EE,\otimes,I,\uHom)$ is a monoid $(A,\mu,\eta)$ equipped with an isomorphism of \emph{right} $A$-modules $\phi:A\cong\uHom(A,I)$. This amounts to an exact pairing $<\!-,-\!>:A\otimes A\to I$ such that $<\!\mu(-,-),-\!>=<\!-,\mu(-,-)\!>$. A Frobenius monoid is \emph{symmetric} if the following equivalent conditions are satisfied:\begin{itemize}\item[(i)]$\phi$ is an isomorphism of $A$-bimodules;\item[(ii)]The map $<\!\mu(-,-),-\!>=<\!-,\mu(-,-)\!>:A\otimes A\otimes A\to I$ is invariant under cyclic permutation of the arguments;\item[(iii)]the exact pairing $<\!-,-\!>:A\otimes A\to I$ is symmetric.\end{itemize}

\begin{lem}[\cite{GK,Me}]\label{endo}The endomorphism operad of a symmetric Frobenius monoid is a multiplicative cyclic operad.\end{lem}

\begin{proof}For a Frobenius monoid $(A,\mu,\eta,\phi)$, $\phi$ induces isomorphisms$$\End_A(n)=\uHom(A^{\otimes n},A)\cong\uHom(A^{\otimes n},\uHom(A,I))\cong\uHom(A^{\otimes n+1},I),\,n\geq 0.$$The cyclic $\ZZ/(n+1)\ZZ$-action on the right endows $\End_A$ with the structure of a cyclic operad. If $A$ is symmetric, its multiplication $\mu_2:I\to\End_A(2)$ is $\ZZ/3\ZZ$-invariant in virtue of (\ref{symmFrob}ii), and similarly $\mu_n:I\to\End_A(n)$ is $\ZZ/(n+1)\ZZ$-invariant, so that $\End_A$ is a multiplicative cyclic operad in the sense of (\ref{cyclicoperad}).\end{proof}

\begin{rmk}It is essential for the validity of Lemma \ref{endo} that a planar, non-symmetric version of cyclic operads is used. Indeed, the multiplicativity of a cyclic operad in Getzler-Kapranov's symmetric sense \cite{GK} would require $\Sg_3$-invariance in (\ref{symmFrob}ii). This extended invariance property holds for \emph{commutative} Frobenius monoids, but in general not for \emph{symmetric} Frobenius monoids.\end{rmk}

The following statement is the cyclic Deligne conjecture for Hochschild cochains; for alternative proofs see \cite{Ka2,KS2,TZ,V2}.

\begin{thm}The normalized Hochschild cochain complex of a symmetric Frobenius algebra carries a canonical action by a framed $E_2$-chain operad.\end{thm}

\begin{proof}The endomorphism operad $\End_A$ of a symmetric Frobenius algebra $A$ is a multiplicative cyclic operad in abelian groups by Lemma \ref{endo}, and is therefore an $\LL_2^{cyc}$-algebra in abelian groups by Proposition \ref{L2cyc}. The unary part of $\LL_2^{cyc}$ endows $\End_A$ with the structure of a cocyclic abelian group. It follows from \cite{GV} that conormalization of the underlying cosimplicial abelian group yields the \emph{normalized Hochschild cochain complex} $\HCC^*(A;A)$ of $A$. By Lemma \ref{leftKan} and (\ref{conorm}) we get$$\uHom_{\Delta C}(\delta_\ZZ^{cyc},\End_A)=\uHom_\Delta(i_!\delta_\ZZ,\End_A)=\uHom_{\Delta}(\delta_\ZZ,i^*\End_A)=\HCC^*(A;A).$$Theorem \ref{cyclcond} induces then through condensation a canonical action on $\HCC^*(A;A)$ by the framed $E_2$-chain operad $\Coend_{\LL_2^{cyc}}(\delta_\ZZ^{cyc})$.\end{proof}

\begin{rmk}\label{CS}The cyclic Deligne conjecture interferes with \emph{string topology} by means of a quasi-isomorphism $N_*(|X|^{|S^1|})\simeq \HCC^*(N^*(X);N_{*}(X))$, cf. \cite[Theorem 3.4.3]{J}, and Poincar\'e duality. Indeed, if $X$ is the singular complex of a $d$-dimensional closed oriented manifold $M$, Poincar\'e duality $N^*(X)\simeq N_{d-*}(X)$ induces a quasi-isomorphism $N_{*+d}(|X|^{|S^1|})\simeq\HCC^*(N^*(X);N^*(X))$. It is conceivable that the Batalin-Vilkovisky structure on Chas-Sullivan's \emph{loop homology} $H_{*+d}(M^{|S^1|})$ stems from a framed $E_2$-action on the Hochschild cochain complex of an appropriate cochain model for $M$.\end{rmk}

\end{document}